\documentclass[reqno,10pt]{amsart}

\title[]{Anosov representations with Lipschitz limit set}

\author[]{Beatrice Pozzetti}

\author[]{Andr\'es Sambarino}

\author[]{Anna Wienhard}
\thanks{}
\date{}

\subjclass[]{}

\usepackage[colorlinks = true,backref = page,hyperindex,breaklinks]{hyperref} 
\usepackage{ stmaryrd }
\usepackage{amssymb}
\usepackage{mathtools}      
\usepackage{mathabx}        
\usepackage[bb = fourier,cal = euler,scr = rsfs]{mathalfa}	
\usepackage{enumitem}       
\usepackage[english]{babel}
\usepackage{tikz}			
\usepackage{tikz-cd}		
\usepackage[font = small]{caption}	
\usepackage{nicefrac}
\usetikzlibrary{calc}
\usepackage[percent]{overpic}
\usepackage{mathdots}

\renewcommand*{\backref}[1]{}
\renewcommand*{\backrefalt}[4]{\quad \tiny
  \ifcase #1 (\textbf{NOT CITED.})%
  \or    (Cited on page~#2.)%
  \else   (Cited on pages~#2.)%
  \fi}

\makeatletter

\def\MRbibitem{\@ifnextchar[\my@lbibitem\my@bibitem}

\def\mybiblabel#1#2{\@biblabel{{\hyperref{http://www.ams.org/mathscinet-getitem?mr=#1}{}{}{#2}}}}

\def\myhyperanchor#1{\Hy@raisedlink{\hyper@anchorstart{cite.#1}\hyper@anchorend}}

\def\my@lbibitem[#1]#2#3#4\par{%
  \item[\mybiblabel{#2}{#1}\myhyperanchor{#3}\hfill]#4%
  \@ifundefined{ifbackrefparscan}{}{\BR@backref{#3}}%
  \if@filesw{\let\protect\noexpand\immediate
    \write\@auxout{\string\bibcite{#3}{#1}}}\fi\ignorespaces%
}

\def\my@bibitem#1#2#3\par{%
  \refstepcounter\@listctr
  \item[\mybiblabel{#1}{\the\value\@listctr}\myhyperanchor{#2}\hfill]#3%
  \@ifundefined{ifbackrefparscan}{}{\BR@backref{#2}}%
  \if@filesw\immediate\write\@auxout
    {\string\bibcite{#2}{\the\value\@listctr}}\fi\ignorespaces%
}

\makeatother

\newcommand{\xqedhere}[2]{%
  \rlap{\hbox to#1{\hfil\llap{\ensuremath{#2}}}}}

\newcommand{\Z}{\mathbb{Z}} 

\newcommand{\R}{\mathbb{R}} \newcommand{\RR}{\R}
\newcommand{\C}{\mathbb{C}}
\newcommand{\N}{\mathbb{N}}
\renewcommand{\P}{\mathbb{P}}

\newcommand{\jac}{\mathcal J}
\newcommand{\sroot}{{\sf{a}}}

\newcommand{\lb}{\llbracket}
\newcommand{\rb}{\rrbracket}
\newcommand{\eps}{\varepsilon}
\newcommand{\G}{\sf{\Gamma}}    
\newcommand{\Gr}{\cal G}    
\newcommand{\cone}{{\cal C}}
\newcommand{\<}{\langle}
\renewcommand{\>}{\rangle}
\newcommand{\E}{\Sigma}
\newcommand{\g}{\gamma}
\newcommand{\z}{\zeta}

\newcommand{\bord}{\partial}
\newcommand{\om}{\omega}
\newcommand{\posgen}{\cal F^{(2)}}
\renewcommand{\t}{\Theta}

\renewcommand{\L}{\Lambda}

\newcommand{\wk}{\check}
\renewcommand{\aa}{{\sf{a}}}
\newcommand{\bs}{{\sf{b}}}
\newcommand{\bb}{{\sf{b}}}

\newcommand{\bordcone}{\cone_\infty}

\newcommand{\scr}{\mathscr}
\renewcommand{\sf}[1]{{\mathsf{#1}}}

\newcommand{\cal}{\mathcal}
\renewcommand{\frak}{\mathfrak}
\newcommand{\sfG}{{\sf G}}

\newcommand{\EE}{{\sf{E}}}

\newcommand{\Weyl}{W}

\DeclareMathOperator{\bus}{b}
\DeclareMathOperator{\ii}{i}
\DeclareMathOperator{\spa}{span}
\DeclareMathOperator{\class}{C}
\DeclareMathOperator{\diam}{diam}
\DeclareMathOperator{\diag}{diag}
\DeclareMathOperator{\SL}{{\mathsf{SL}}}
\DeclareMathOperator{\PSL}{{\mathsf{PSL}}}
\DeclareMathOperator{\GL}{{\mathsf{GL}}}
\DeclareMathOperator{\SO}{{\mathsf{SO}}}
\DeclareMathOperator{\PGL}{{\mathsf{PGL}}}
\DeclareMathOperator{\PO}{{\mathsf{PO}}}

\DeclareMathOperator{\Hff}{dim_{Hf{}f}}

\DeclareMathOperator{\ann}{Ann}

\newcommand{\cQ}{\Psi}
\newcommand{\cD}{\cal D}

\newcommand{\cartan}{a}

\newcommand{\Id}{\mathrm{Id}}
\newcommand{\st}{\,\mathord{\colon}\,} 

\renewcommand{\angle}{\measuredangle}
\newcommand{\Wedge}{\mathsf{\Lambda}}  

\newcommand{\Ll}{\mathscr L} 
\DeclareMathOperator{\Sp}{{\sf{Sp}}}
\DeclareMathOperator{\SU}{{\sf{SU}}}
\newcommand{\bS}{\mathbb S}

\newcommand{\K}{\mathbb K}
\newcommand{\HH}{\mathbb H}
\DeclareMathOperator{\PSp}{{\sf{PSp}}}

\newcommand{\Aff}{{\sf{Aff}}}
\newcommand{\calO}{\mathcal O}
\newcommand{\bH}{\mathbb H}
\DeclareMathOperator{\AdS}{AdS}

\DeclareMathOperator{\Is}{{\sf{Is}}} 
\DeclareMathOperator{\Span}{Span}
\DeclareMathOperator{\Stab}{Stab}
\newcommand{\rg}{0}
\newcommand{\bpm}{\begin{pmatrix}}
\newcommand{\epm}{\end{pmatrix}}

\renewcommand{\epsilon}{\varepsilon}

\hypersetup{bookmarksdepth  =  3} 
\numberwithin{equation}{section}     

\setlist[enumerate,1]{label = {\upshape(\roman*)},ref = \roman*}
\setlist[enumerate,2]{label = {\upshape(\alph*)},ref = \alph*}

\newtheorem{thmA}{Theorem}

\newtheorem{thm}{Theorem}[section]

\newtheorem{cor}[thm]{Corollary}
\newtheorem{lemma}[thm]{Lemma}
\newtheorem{prop}[thm]{Proposition}

\theoremstyle{definition}

\newtheorem{defi}[thm]{Definition}

\newtheorem{remark}[thm]{Remark}
\newtheorem{assumption}[thm]{Assumption}
\newtheorem{ex}[thm]{Example}

\theoremstyle{remark}
\newtheorem{obs}[thm]{Remark}

\thanks{A.S. was partially financed by ANR DynGeo ANR-16-CE40-0025. B.P. and A.W acknowledge funding by the Deutsche Forschungsgemeinschaft within the Priority Program SPP 2026 	“Geometry at Infinity”, and from the National Science Foundation under Grant No. 1440140, while they were in residence at the Mathematical Sciences Research Institute in Berkeley, California, during the semester of fall 2019. A.W. acknowledges funding by the European Research Council under 	ERC-Consolidator grant 614733, and by the Klaus-Tschira-Foundation. This work is supported by the Deutsche Forschungsgemeinschaft under Germany’s Excellence Strategy EXC-2181/1 - 390900948 (the Heidelberg STRUCTURES Cluster of Excellence).}
\begin{document}

\begin{abstract}
We study Anosov representations whose limit set has intermediate regularity, namely is a Lipschitz submanifold of a flag manifold. We introduce an explicit linear functional, the unstable Jacobian, whose orbit growth rate is integral on this class of representations. 
We prove that many interesting higher rank representations,  including $\t$-positive representations, belong to this class, and establish several applications to rigidity results on the orbit growth rate in the symmetric space. 
\end{abstract}

\maketitle

\tableofcontents

\section{Introduction}
Let $\G\subset\PGL_d(\R)$ be a discrete subgroup. Following Guivarc'h, Benoist \cite{limite} has shown that if $\G$ contains a proximal element and acts irreducibly on $\R^d$ then its action on projective space $\P(\R^d)$ has a smallest closed invariant set. This is usually called \emph{Benoist's limit set} or simply \emph{the limit set} of $\G$ on $\P(\R^d)$ and denoted by $\sf L_\G.$

In contrast with the negatively curved situation, the limit set of a subgroup $\G$ whose Zariski closure has rank $\geq2$ needs not be a fractal object. Examples of  infinite co-volume Zariski-dense groups whose limit set is a proper $\class^1$ submanifold  arise in the study of strictly convex divisible sets (Benoist \cite{convexes1}) and of Hitchin representations (Labourie \cite{labourie}). Lately, more examples of subgroups with this property were found by P.-S.-W. \cite{PSW1} and Zhang-Zimmer \cite{ZZ}. 

Intermediate phenomena also occur:  the limit set of the direct sum $(\rho,\eta):\pi_1S\to\PSL_2(\R)\times\PSL_2(\R)$  of the holonomies of two hyperbolizations of a closed topological surface $S$ is a Lipschitz circle that is never $\class^1$;  the same occurs more generally for maximal representations (Burger-Iozzi-W. {\cite{MaxReps}}), Quasi-Fuchsian AdS representations (Barbot--M\'erigot \cite{BarbotMerigot}), and $\HH^{p,q}$ convex-cocompact representations (Danciger--Gu\'eritaud--Kassel \cite{DGKcc}). 

This paper provides the first systematic investigation of this intermediate phenomenon - its main object are discrete groups whose limit set is a Lipschitz manifold. We will restrict our investigation to the class of Anosov subgroups, a robust and rich class of strongly undistorted subgroups of semisimple Lie groups (see Section \ref{AnosovSL} for the precise definition).

For discrete subgroups $\G$ of $\SO(1,n)$ Sullivan \cite{sullivan} established a beautiful relation between a geometric invariant of the limit set $\sf L_\G$, its Hausdorff dimension, and a dynamical invariant for the action of $\G$ on the symmetric space $\HH^n$, the orbit growth rate. This was further used by Bowen \cite{bowen} to prove a strong rigidity result: for fundamental groups of surfaces acting on $\HH^3,$ the Hausdorff dimension of the limit set is minimal if and only if the limit set is $\class^1$ and $\G$ preserves a totally geodesic copy of $\HH^2$ on which it acts cocompactly. When $\sf G$ has higher rank, the situation is more complicated as one can additionally consider orbit growth rates with respect to different linear functionals $\varphi$ (as in, for example, Quint \cite{quint2}). It is a challenging problem to understand what are the functionals $\varphi$ whose orbit growth rate carries geometric information on the group $\G$ or on its limit set $\sf L_\G$.

The main contribution of the paper is to single out an explicit linear functional, the \emph{unstable Jacobian}, whose critical exponent is integral on Anosov subgroups whose limit set is a Lipschitz submanifold. In order to prove such a result, we import ideas from non-conformal dynamics, such as the study of the affinity exponent, to  the setting of Anosov groups, and use the Anosov property, together with ideas from geometric group theory, to establish a strengthening of the theory of Patterson-Sullivan denisities developed by Quint; these two results are of independent interest. We then showcase the strength of our main result by applying it to several well studied classes of representations: maximal representations, $\HH^{p,q}$-convex cocompact subgroups and $\t$-positive representations. 

\subsection*{The unstable Jacobian and the affinity exponent}
We now introduce some notation useful to explain more precisely our results.
We denote by 
$$\sf E=\{\underline a=(a_1,\ldots,a_d)\in\R^d:\sum_i a_i=0\}$$
the Cartan subspace of the Lie group $\PGL_d(\R)$, by $$\aa_i(\underline a)=a_i-a_{i+1}$$ the $i$-th simple root and by $\sf E^+\subset \sf E$ the Weyl chamber whose associated set of simple roots is $\Pi=\{\aa_i:i\in\lb1,d-1\rb\}.$ Let $\cartan:\PGL_d(\R)\to\sf E^+$ be the \emph{Cartan projection} with respect to the choice of a scalar product $\tau$.  Concretely
$\cartan(g)=\left(\log\sigma_1(g),\ldots,\log\sigma_d(g)\right),$
where $\sigma_i(g)$ denote the \emph{singular values} of the matrix $g$,  the square roots of the eigenvalues of the matrix $gg^*,$ where $g^*$ is the adjoint operator of $g$ with respect to $\tau$. 

Given a discrete subgroup $\G<\PGL_d(\R),$ the \emph{critical exponent} of a linear form $\varphi\in\sf E^*$, denoted by $h_\G(\varphi)$, 
is defined as 
$$h_\G(\varphi):=\lim_{T\to \infty}\frac{\log \#\big\{\g\in\G|\; \varphi\big(\cartan (\g)\big)<T\big\}}{T}.$$
In this paper we introduce the  \emph{$p$-th unstable Jacobian} $\jac^u_p\in\sf E^*$ defined by 
$$\jac^u_p=(p+1)\omega_{\aa_1}-\omega_{\aa_{p+1}},$$
where $\omega_{\aa_p}(\underline a)=\sum_1^pa_i$ is the fundamental weight relative to the $p$-th simple root $\sroot_p$.   Our main result is
\begin{thmA}\label{Lipschitz}
 Let $\G<\PSL_d(\R)$ be a strongly irreducible, projective Anosov subgroup whose limit set  $\sf L_\G<\P(\R^d)$ is  a Lipschitz submanifold of dimension $p$. Then  
$$h_\G(\jac^u_p)=1.$$ 
If $p=1$ the same holds replacing strong irreducibility with weak irreducibility\footnote{We say that a subgroup $\G<\PSL_d(\R)$  is \emph{weakly irreducible} if the vector space $\spa\big(\sf L_\G\big)$ is $\R^d.$}.
\end{thmA}
A similar result was proven, in the context of fundametal groups of compact strictly convex projective manifolds by Potrie-S. \cite[Theorem B]{exponentecritico}; our approach is entirely different and, since we require less regularity, its scope of application is considerable broader. Note that up to postcomposing with a suitable linear representation, any Anosov representation can be turned in a projective Anosov representation.

We prove the two inequalities in Theorem \ref{Lipschitz} as corollary of two different results that are applicable in more general settings. We  focus first on the lower bound on the critical exponent (Corollary \ref{cor:1.3}) that follows from a general result on the Hausdorff dimension of limit sets (of projective Anosov representations). 

An important step in the proof is the study, in the context of Anosov representations, of the \emph{affinity exponent}: a key notion from non-conformal dynamics that first appeared in Kaplan-Yorke \cite{KaYo} and Douady-Oesterlé \cite{DoOe}, and played a prominent role in Falconer's work \cite{Falconer}. 
More specifically, for every discrete subgroup $\G<\PSL_d(\R)$, we consider  the \emph{piecewise} Dirichlet series defined, for $p\in\N$ and $s\in[p-1,p],$ by 
$$\Phi_\G^\Aff(s)=\sum_{\g\in\G}\left(\frac{\sigma_2}{\sigma_1}(\g)\cdots\frac{\sigma_{p}}{\sigma_1}(\g)\right) \left(\frac{\sigma_{p+1}}{\sigma_1}(\g)\right)^{s-(p-1)}:s\in[p-1,p].$$
The affinity exponent is the critical exponent of this series:
 $$h^\Aff_\G\colon=\inf\big\{s:\Phi_\G^\Aff(s)<\infty\big\}=\sup\big\{s:\Phi_\G^\Aff(s)=\infty\big\}\in(0,\infty].$$

The second main result of this paper is the following (see \S\ref{sec:2} for a statement for arbitrary local fields):
\begin{thmA}\label{Hffand1st}
	Let $\G<\PGL_d(\R)$ be  projective Anosov, then 
	$$\Hff\big(\sf L_\G\big)\leq h^\Aff_\rho.$$
\end{thmA}

It is easy to deduce from Theorem \ref{Hffand1st}  relations between the Hausdorff dimension of the limit set of a projective Anosov subgroup and the orbit growth rate with respect to explicit linear functionals on the Weyl chamber.  Since the quantity  $h_\G(\jac_{p-1}^u)$ appearing in Theorem \ref{Lipschitz} is also the critical exponent of the Dirichlet series 
$$s\mapsto\sum_{\g\in\G} \left(\frac{\sigma_1\cdots\sigma_p}{\sigma_1^p}(\g)\right)^s, $$
we get:
\begin{cor}\label{cor:1.3}
	Let $\G<\PGL_d(\R)$ be projective Anosov and assume furthermore that $\Hff(\sf L_\G))\geq p.$ Then 
	$$\Hff(\sf L_\G)\leq ph_\rho(\jac^u_p).$$
\end{cor}
Observe that $\jac^u_1=\aa_1,$ and thus, whenever $\Hff(\sf L_\G)\geq 1$ we obtain as a consequence the results of Glorieux-Monclair-Tholozan \cite[Theorem 4.1]{GMT} and P.-S.-W. \cite[Proposition 4.1]{PSW1}.

\subsection*{Existence of Patterson Sullivan measures}\label{second}
The second inequality in Theorem \ref{Lipschitz} follows from an improvement on a result by Quint \cite[Th\'eor\`eme 8.1]{quint1}  concerning the relation between critical exponents and the existence of $(\G,\varphi)$-Patterson-Sullivan measures. 

Given a set $\t\subset\Pi$ of simple roots, we denote by $\cal F_\t$ the associated partial flag manifold, this consists of the space of  flags of subspaces of dimension  indexed by $\t.$ 
We denote by $\sf E_\t$ the  Levi subspace  of $\sf E$ defined by 
$$\sf E_\t=\bigcap_{p\notin\t}\ker\aa_p.$$ 
The restrictions of the fundamental weights $\{\omega_{\aa_p}|_{\sf E_\t}:p\in\t\}$ span its dual $(\sf E_\t)^*$.
Using the Iwasawa decomposition of $\PGL_d(\R)$, Quint introduced an \emph{Iwasawa cocycle}
$$\bus_\t:\PGL_d(\R)\times\cal F_\t\to\sf E_\t$$
that is the higher rank analog of the more studied \emph{Busemann cocycle} in negative curvature (see  Quint \cite[Lemma 6.6]{quint1} and Section \ref{s.PS} for the precise definition). With this notation at hand we can recall the definition of a $(\G,\varphi)$-Patterson-Sullivan measure from \cite{quint1}.

\begin{defi}Given a discrete subgroup $\G<\PGL_d(\R)$ and $\varphi\in(\sf E_\t)^*$ a \emph{$(\G,\varphi)$-Patterson-Sullivan measure} on $\cal F_\t$ is a finite Radon measure $\mu$ such that for every $g\in\G$ one has $$\frac{\sf d g_*\mu}{\sf d\mu}(x)=e^{-\varphi\big(\bus_\t(g^{-1},x)\big)}.$$ 
\end{defi}

Inspired by a classical result by Sullivan \cite{sullivan}, Quint shows  \cite[Th\'eor\`eme 8.1]{quint1} that the existence of a $(\G,\varphi)$-Patterson-Sullivan measure on $\cal F_\t$ gives an upper bound on a related critical exponent \begin{equation}\label{q}h_\G(\varphi+\rho_{\theta^c})\leq 1.\end{equation}
Here $\rho_{\theta^c}$ is an explicit linear functional which is positive on the interior of the Weyl chamber and accounts for the possible growth along the fibers of the projection $\cal F_\Delta\to\cal F_\t$ \cite[Lemme 8.3]{quint1}. In general $h_\G(\varphi+\rho_{\theta^c})<h_\G(\varphi)$ and thus Quint's result is not sharp enough for our purposes. 
Using ideas from geometric group theory we show that, provided the group $\G$ is Anosov with respect to one of the roots in $\t$, there is no contribution from the fibers:
\begin{thmA}\label{thmC} 
	Let $\G<\PGL_d(\R)$ be  projective Anosov and consider $\Theta\subset\Pi$ such that $\sroot_1\in\t.$ Let $\varphi\in(\sf E_\t)^*.$ If there exists a $(\G,\varphi)$-Patterson-Sullivan measure on $\cal F_\t$ whose support is not contained on a complementary subspace\footnote{Given $\t\subset\Pi$ denote by $\ii\t=\{d-p:p\in\t\}.$ Two points $(x,y)\in\cal F_\t\times\cal F_{\ii\t},$ are \emph{transverse} if for every $p\in\t$ one has that $x^p\cap y^{d-p}=\{0\}.$ A \emph{complementary subspace of $\cal F_\t$} is a subset of $\cal F_\t$ of the form $$\{x\in\cal F_\t: x\textrm{ is not transverse to }y_0\}$$ for a given $y_0\in\cal F_{\ii\t}.$}, then 
	$$h_\G(\varphi)\leq1.$$
\end{thmA}
We refer the reader to \S \ref{Lip} and Theorem \ref{phiInD} for a version of Theorem \ref{thmC} where the target group is an arbitrary semi-simple group over a local field.

We provide the link between Theorem \ref{thmC} and Theorem \ref{Lipschitz} in Section \ref{AnosovLipschitz},
where we establish that, if $\G<\PSL_d(\R)$ is a projective Anosov subgroup whose limit set $\sf L_\G$ is a Lipschitz submanifold of dimension $p$ then there exists a $(\G,\jac_{p}^u)$ Patterson-Sullivan measure on $\cal F_{\{\aa_1,\aa_p\}}$. In fact we explicitly construct such a measure using Rademacher's Theorem and an 
explicit volume form on the almost everywhere defined tangent space to $\sf L_\G$ (Proposition \ref{quasi-invariance}).

\begin{ex}If $\rho:\pi_1S\to\PSp(4,\R)$ is a maximal representation (see \S \ref{sec:max} for the definition), the combination of Theorems \ref{Hffand1st} and \ref{thmC} gives $h_{\rho(\pi_1S)}(\aa_2)=1$ while Quint's result (equation (\ref{q})) becomes $h_{\rho(\pi_1S)}(\omega_{\aa_2})\leq1.$ This latter inequality is implied by the former equality, and often far from being sharp: one can find representations $\rho$ for which $h_{\rho(\pi_1S)}(\omega_{\aa_2})$ is arbitrarily small.
\end{ex}

Theorem \ref{thmC} is complementary and independent from the Patterson-Sullivan theory for Anosov representations developed by Dey--Kapovich \cite{DK}. They only consider Patterson-Sullivan densities with respect to functionals $\varphi$ that, as opposed to the unstable Jacobian, belong to $(\sf E_\theta)^*$ where the representation is assumed to be Anosov with respect to \emph{all} elements of $\theta,$ and induce Finsler distances on the symmetric space (see also Ledrappier \cite{ledrappier} for a different approach yielding similar results); a drawback of their approach is that they can only relate the critical exponent with a pre-metric induced from a Finsler distance on the symmetric space that is hard to compute. On the opposite we begin with a natural measure supported on the limit set, which belongs to the Lebesgue measure class, find a suitable functional, the unstable Jacobian,  turning such measure into a Patterson-Sullivan measure, and deduce from this geometric properties of the action of $\G$ on the symmetric space.

\subsection*{Intermediate regularity and $\class^1$-dichotomy}
The class of Anosov subgroups with Lipschitz limit sets is very rich, and includes the images of many well studied classes of representations such as maximal representations  (Burger--Iozzi--W. \cite{MaxReps}, see also \S \ref{sec:max}), Quasi-Fuchsian AdS representations (Barbot--M\'erigot \cite{BarbotMerigot}) and $\HH^{p,q}$-convex cocompact representations (Danciger--Gu\'eritaud--Kassel \cite{DGKcc}, see also \S \ref{sec:5}).  

As another contribution of the paper of independent interest we show that also $\Theta$-positive representations of fundamental groups of surfaces in $\SO(p,q)$ (Guichard-W. \cite{GWpositivity}) yield subgroups with this property. We refer the reader to \S \ref{sec:positive} for the precise definition of $\Theta$-positive representations. We will only\footnote{Guichard-Labourie-W. announced that all $\t$-positive representations are $\Theta$-Anosov, so this should not pose any restriction.} consider here the $\Theta$-positive representations that are furthermore $\Theta$-Anosov for $\Theta=\{\aa_1,\ldots,\aa_{p-1}\}$, as a result, for each $k\in\Theta$, they admit a boundary map $\xi^k:\bord\G\to\Is_k(\R^{p,q})$ parametrizing the limit set in the Grassmannian of $k$-dimensional isotropic subspaces. In Section \ref{sec:positive} we prove:
\begin{thmA}\label{thm:pos}
	Let $\rho:\G\to\SO(p,q)$ be a $\Theta$-Anosov representation that is $\Theta$-positive. Then the images of the boundary maps $\xi^k:\bord\G\to \Is_k(\R^{p,q})$ are $\class^1$ submanifolds for each $1\leq k< p-1$, and the image $\xi^{p-1}(\bord\G)$ is Lipschitz.
\end{thmA}

We will prove the two parts of Theorem \ref{thm:pos} separately, respectively in Corollary \ref{c.c1} and Proposition \ref{p.lipschitz}.

At least for representations of fundamental groups of surfaces, the regularity of the limit set on a given (maximal) flag space seems to be related to the position of the associated root among the Anosov roots. By definition, a simple root is an \emph{Anosov root} (for a subgroup $\G$) if its kernel intersects trivially the limit cone $\cal L_\G$ of $\G.$ Among such roots one can consider the internal ones, i.e. such that every neighboring root in the Dynkin diagram is also an Anosov root, or the ones in the boundary, i.e. connected to a root that non-trivially intersects $\cal L_\G.$ For example, for a $\Theta$-positive representation in $\SO(p,q),$ the roots $\{\aa_1,\ldots,\aa_{p-2}\}$ are internal, while $\aa_{p-1}$ is the only boundary root.

The intermediate regularity (Lipschitz but not $\class^1$) of limit sets for surface groups seems only to occur for boundary roots. For internal roots, we can prove a $\class^1$-dichotomy ruling out intermediate regularity in several interesting cases. More specifically we consider fundamental groups $\G$ of compact surfaces and study small deformations of representations of the form 
$$\G \to \PSL_2(\R)\overset R \to \PSL_d(\R),$$ 
that are $\{\aa_1,\aa_2\}$-Anosov (this latter assumption can be rephrased as a proximality assumption on the linear representation $R$).  For any such representation we have an explicit dichotomy:  the associated limit set is either $\class^1$ or not even Lipschitz (Corollary \ref{c.dic_intro}). We refer the reader  to Section \ref{S.C^1} for the precise statement of the dichotomy.

\subsection*{Entropy rigidity results}
We conclude the introduction by discussing three well studied classes of representations to which Theorem \ref{Lipschitz} applies. Interestingly, in all these cases, the mere information on the critical exponent of the unstable Jacobian provided by Theorem \ref{Lipschitz}, allows to obtain a sharp upper bound on the critical exponent for the action on the symmetric space endowed with the Riemannian (!) distance function. In the case of $\Theta$-positive representations this is even sufficient to prove that the bound is rigid: it is attained only on the specific Fuchsian locus, the generalization, in our setting, of Bowen's aforementioned result.

\subsection*{Maximal representations} Maximal representations are well studied representations of fundamental groups of surfaces in Hermitian Lie groups $\sf G_\R$ that were introduced by Burger--Iozzi--W. \cite{MaxReps} through a cohomological invariant, the Toledo invariant. For these representations Theoem \ref{Lipschitz} applies and gives

\begin{thm}\label{thm:max_intro}
Let $\sf G_\R$ be a classical simple Hermitian Lie group of tube type. 
Let $\rho:\G\to\sf G_\R$ be a maximal representation, and let $\check \aa$ denote the root associated to the stabilizer of a point in the Shilov boundary of $\sf G_\R$. Then $h_\rho(\check{\aa}) = 1.$ 
\end{thm} 
 Concretely, in the case $\sf G_\R\in\{\Sp(2p,\R),\SU(p,p),\SO^*(4p)\}$ the root $\check{\aa}$ computes the logarithm of the square of the middle eigenvalue, while for  $\sf G=\SO_0(2,p)$ the root $\check{\aa}$ is the first root, computing the logarithm of the first eigenvalue gap.

Theorem~\ref{thm:max_intro} also holds for the exceptional Hermitian Lie group of tube type if the representation is Zariski-dense, and we expect it to hold unconditionally. 
We refer the reader to \S \ref{sec:max} for  a slightly more general statement, further explanations and consequences, in particular concerning a sharp upper bound on the exponential orbit growth rate for the action on the symmetric space (see Proposition \ref{prop:hmax}).

\subsection*{$\bH^{p,q}$-convex-cocompact representations} 

Generalizing work of Mess \cite{mess} and Barbot-M\'erigot \cite{BarbotMerigot}, Danciger--Gu\'eritaud--Kassel \cite{DGKcc} introduced a class of representations called \emph{$\bH^{p,q}$-convex cocompact}. Here $\bH^{p,q}$ is the \emph{pseudo-Riemannian hyperbolic space}, consisting of negative lines in $\P(\R^d)$  for a fixed non-degenerate form $Q$ of signature $(p,q+1)$. It follows then from {\cite[Theorem 1.11]{DGKcc}}, that a projective Anosov subgroup $\G<\PO(Q)=\PO(p,q+1)$ is \emph{$\bH^{p,q}$-convex cocompact} if for every pairwise distinct triple of points $x,y,z\in \sf L_\G$, the restriction $Q|_{\langle x, y, z\rangle}$ has signature $(2,1)$.

Consider a representation $\Lambda:\PO(p,1)\to\PO(p,q+1)$ whose image stabilizes a $(p+1)$-dimensional subspace $V$ of $\R^d$ where $Q|_V$ has signature $(p,1).$ Endow the symmetric space $X_{p,q+1}$ with the $\PO(p,q+1)$-invariant Riemannian metric normalized so that the totally geodesic copy of $\bH^p$ in $X_{p,q+1}$ stabilized be $\Lambda$ has constant curvature $-1.$ 
\begin{defi}\label{expcr}For a subgroup $\G<\SO(p,q+1)$ and $x_0\in X_{p,q+1}$ denote by $h_\rho^{X_{p,q+1}}$ the critical exponent of the Dirichlet series 
$$s\mapsto \sum_{\g\in\G}e^{-sd(x_0,\rho(\g)x_0)}.$$
\end{defi}
We have the following upper bound. 

\begin{prop}
Assume that $\bord\G$ is homeomorphic to a $(p-1)$-dimensional sphere and let $\G<\PO(p,q+1)$ be strongly irreducible and $\bH^{p,q}$-convex-cocompact, then $$h_\rho^{X_{p,q+1}}\leq p-1.$$
\end{prop}

We expect this upper bound to be rigid, namely the upper bound should only be attained at an inclusion of a co-compact lattice in $\PO(p,1)$  preserving a totally geodesic copy of $\bH^p$ of the type induced by $\L.$ However, only the case $p=2$ is known due to Collier-Tholozan-Toulisse \cite{CTT}.

Section \ref{sec:5} contains more information on $\bH^{p,q}$-convex cocompact representations, in particular the relation with recent work by Glorieux-Monclair \cite{GMcc}.

\subsection*{$\Theta$-positive representations}  
Thanks to Theorem \ref{thm:pos}, Theorem \ref{Lipschitz} also applies to $\Theta$-positive representations of fundamental groups of surfaces in $\SO(p,q)$ and gives:

\begin{cor}\label{c.tp}
Let $\rho:\G\to\SO(p,q)$ be a $\Theta$-Anosov representation that is $\Theta$-positive and weakly irreducible, then $h_\rho(\sroot_k)=1$ for every $k\leq p-1$.
\end{cor}

Inspired from Potrie-S. \cite{exponentecritico}, we deduce from Corollary \ref{c.tp} a rigid upper bound for the critical exponent of the action of a positive representation on the Riemannian symmetric space $X_{p,q}$ (see Theorem \ref{prop:hpositive}). More precisely, we now normalize  the $\SO(p,q)$-invariant Riemannian metric on $X_{p,q}$ so that the totally geodesic copy of $\bH^2$ induced by the representation $\Wedge:\SL_2(\R)\to\SO(p,q)$ that stabilizes a subspace of $\R^d$ of signature $(p,p-1),$ has constant curvature -1. We consider the critical exponent in Definition \ref{expcr} with this normalization of distance.

\begin{thm}
Let $\G$ be the fundamental group of a surface and let $\rho:\G\to\SO(p,q)$ be $\t$-positive. Then the critical exponent with respect to the Riemannian metric satisfies $$h_\rho^{X_{p,q}}\leq1.$$ Furthermore if equality is achieved at a totally reducible representation $\eta$, then $\eta$ splits as $W\oplus V$,  $W$ has signature $(p,p-1)$, $\eta|_W$ has Zariski closure the irreducible $\PO(2,1)$ in $\PO(p,p-1)$, and $\eta|_V$ lies in a compact group. 
\end{thm}

New  arguments are needed with respect to \cite{exponentecritico} since the Anosov-Levi space of a $\t$-positive representation has codimension one (instead of 0, which is the case treated in \cite{exponentecritico}), see \S \ref{sec:positive}.

\subsection{Plan of the paper} 
In \S \ref{preliminaries} we introduce some required preliminaries, and recall some needed results from Bochi-Potrie-S. \cite{BPS} and P.-S.-W. \cite{PSW1}. Section \ref{sec:2} deals with the affinity exponent and Hausdorff dimension for Anosov representations, in it we prove Theorem \ref{Hffand1st} for any local field.
Section  \ref{semi-simplegroups} is a reminder on (more or less) standard definitions on semi-simple algebraic groups over a local field. 
In \S \ref{Lip} we recall objets from higher rank Patterson-Sullivan Theory and in subsection \ref{s.PS} we prove Theorem \ref{phiInD} (a broader version of Theorem \ref{thmC}).
Section \S \ref{AnosovLipschitz} completes the proof of Theorem \ref{Lipschitz}. The remaining sections deal with the applications of this result discussed  in the introduction.
\subsection*{Acknowledgements}
We would like to thank J.-F. Quint for pointing us to Falconer's work and suggesting to consider the affinity exponent.

\section{Preliminaries}\label{preliminaries}
We recall in this section the notions we will need concerning Anosov representations and cone types. We refer the reader to \cite{PSW1} and the references therein for more details.

Throughout the paper $\K$ will denote a  local field with absolute value $|\cdot|:\K\to \R^+$. If $\K$ is non-Archimedean, we require that $|\omega|=\frac 1q$ where $\omega$ denotes the \emph{uniformizing element}, namely a generator of the maximal ideal of the valuation ring $\calO$, and $q$ is the cardinality of the residue field $\calO/\omega\calO$ (this is finite because $\K$ is, by assumption, local). This guarantees that the Hausdorff dimension of $\P^1(\K)$ equals $1$.
\subsection{Singular values and Anosov representations into $\PGL_d(V_\K)$}\label{ss.dom_sing}

A $\K$-norm $\|\,\|$ on a $\K$ vector space $V_\K$ induces a norm  on every exterior power of $V$; the angle between two vectors $\angle(v,w)$ is the unique number in $[0,\pi]$ such that 
$$\sin\angle(v,w):=\frac{\|v\wedge w\|}{\|v\|\|w\|}$$
Given two points $[v],[w]\in\P (V)$, we define their distance as
$$d([v],[w]):=\sin \angle(v,w),$$
and given any two subspaces $P,Q<V$ we define their minimal angle as 
$$\angle(P,Q)=\min_{v\in P\setminus\{0\}}\min_{w\in Q\setminus\{0\}}\angle(v,w).$$

An element $a\in\GL(V_\K)$ is a \emph{semi-homothecy} (for a norm $\|\cdot\|$) if there exists a $a$-invariant $\K$-orthogonal\footnote{Recall that for $\K$ non-Archimedean a decomposition $V=V_1\oplus\cdots\oplus V_k$ is \emph{orthogonal} if, for every $v_i\in V_i$, it holds $\|\sum v_i\|=\max_i\|v_i\|$.} decomposition $V=V_1\oplus\cdots\oplus V_k$ and $\sigma_1,\cdots,\sigma_k\in\R_+$ such that for every $i\in\lb1,k\rb$ and every $v_i\in V_i$ one has $$\|av_i\|=\sigma_i\|v_i\|.$$ The numbers $\sigma_i$ are called the ratios of the semi-homothecy $a.$

Following Quint \cite[Th\'eor\`eme 6.1]{Quint-localFields}, we fix a maximal abelian subgroup of diagonalizable matrices $A\subset\GL(V_\K),$ a compact subgroup $K\subset\GL(V_\K)$ such that, if $N$ is the normalizer of $A$ in $\GL(V_\K)$, then $N=(N\cap K)A$, and a $\K$-norm $\|\,\|$ on $V$ preserved by $K,$ and such that $A$ acts on $V$ by semi-homothecies.  Let  $e_1\oplus\cdots\oplus e_d$ be the eigenlines of $A$ (here $d=\dim V$) and choose the Weyl chamber $A^+$ consistsing of those elements $a\in A$ whose corresponding semi-homothecy ratios verify $\sigma_1(a)\geq\cdots\geq\sigma_d(a).$

For every $g\in \GL(V_\K)$ we choose a Cartan decomposition $g=k_ga_gl_g$ with $a_g$ in $A^+$, $k_g,l_g\in K$, and denote by  $$\sigma_1(g) \ge \sigma_2(g) \ge \cdots \ge \sigma_d(g) $$ the semi-homotecy ratios of the Cartan projection $a_g\in A^+$ (these do not depend on the choice of the Cartan decomposition once $K$ and $\|\cdot\|$ are fixed). In order to simplify notation we will often write $\frac{\sigma_i}{\sigma_j}(g)=\frac{\sigma_i(g)}{\sigma_j(g)}.$

We define, for $p \in \lb1,d-1\rb,$  
$$u_p(g)=k_g\cdot e_p\in V.$$
 The set $\{u_p(g):p\in\lb1,d-1\rb\}$ is an \emph{arbitrary} orthogonal choice of the axes (ordered in decreasing length) of the ellipsoid $\{Av \st \|v\| = 1\},$ and, by construction, for every $v\in g^{-1}u_p(g)$ one has $\|gv\|=\sigma_p(g)\|v\|.$
Let $$U_p(g)=u_1(g)\oplus\cdots\oplus u_p(g)=k_g\cdot(e_1\oplus\cdots\oplus e_p).$$
If $g$ is such that $\sigma_p(g) > \sigma_{p+1}(g)$, then we 
say that \emph{$g$ has a gap of index $p$}. In that case the decomposition $$U_{d-p}(g^{-1}) \oplus  g^{-1}(U_p(g))$$ is orthogonal (cfr. \cite[Remark 2.4]{PSW1}) and, if $\K$ is Archimedean, the $p$-dimensional space $U_p(g)$ is independent of the Cartan decomposition of $g.$

We will denote by $\Pi=\{\sroot_1,\ldots, \sroot_{d-1}\}$ the root system of $\PGL(V_\K)$, and, given a subset $\theta\subset \Pi$, by $\cal F_\theta$ the associated partial flag manifold.  
Given $\theta\subset\Pi$ we also denote by $U^\theta(g)$ the partial flag $U^{\theta}(g)=\{U_p(g):\aa_p\in\theta\}.$ The \emph{$\theta$-basin of attraction of $g$}
\begin{equation}\label{Btheta}
B_{\theta,\alpha}(g)=\{x^{\theta}\in\cal F_{\theta}(\K^d):\min_{\aa_p\in\theta}\angle\big(x^{p},U_{d-p}(g^{-1})\big)>\alpha\}
\end{equation}
is the complement of the $\alpha$-neighborhood of  $U^{\theta^c}(g^{-1}).$ When $\theta$ consists of a single root $\sroot$ we will write $B_{\sroot,\alpha}(g)$ instead of $B_{\{\sroot\},\alpha}(g)$
\begin{remark}
If $g$ has a gap of index $p$, then  $U_{d-p}(g^{-1})$ is well defined if $\K$ is Archimedean, and any two possible choices have distance at most $\frac{\sigma_{p+1}}{\sigma_{p}}(g)$ if $\K$ is non-Archimedean. It follows that, also in the non-Archimedean case,
$B_{\theta,\alpha}(g)$ only depends on $K$ provided $\alpha$ is bigger than the minimal singular value gap.
\end{remark}

We recall for later use the following lemma, which explains the choice of the term basin of attraction: 
\begin{lemma}[{Bochi-Potrie-S. \cite[Lemma A.6]{BPS}}]\label{l.basin}
For every $g\in \PGL_d(\K)$, and $x\in B_{\sroot_1,\alpha}(g)$ it holds
$$d(U_1(g),g\cdot x)\leq \frac1{\sin(\alpha)} \frac{\sigma_2}{\sigma_1}(g).$$
\end{lemma}

\subsection{Anosov representations}\label{AnosovSL}
Let $\G$ be a word-hyperbolic group with identity element $e$, fix a finite symmetric generating set $S_\G$. For $\g\in\G-\{e\}$ denote by $|\g|$ the least number of elements of $S_\G$ needed to write $\g$ as a word on $S,$ and define the induced distance $d_\G(\g,\eta)=|\g^{-1}\eta|.$ A geodesic segment on $\G$ is a sequence $\{\alpha_i\}_0^k$ of elements in $\G$ such that $d_\G(\alpha_i,\alpha_j)=|i-j|.$

\begin{defi}\label{defAnosov} A representation $\rho:\G\to\PGL_d(\K)$ is \emph{$\sroot_p$-Anosov}\footnote{In the language of Bochi-Potrie-S. \cite[Section 3.1]{BPS} a $\sroot_p$-Anosov representation  is called \emph{$p$-dominated}. It was proven by  Kapovich-Leeb-Porti \cite{KLP1} that if a group $\G$ admits an Anosov representation, it is necessarily word hyperbolic. See also  Bochi-Potrie-S. \cite{BPS} for a different approach.}
 if  there exist positive constants $c,\mu$, the \emph{$\sroot_p$-Anosov constants of $\rho$}, such that for all $\g\in\G$ one has 
\begin{equation}\label{def1}
\frac{\sigma_{p+1}}{\sigma_{p}}\big(\rho(\g)\big)\leq ce^{-\mu|\g|}.
\end{equation} An $\sroot_1$-Anosov representation will be called \emph{projective Anosov}.
\end{defi}
The following result was proven in Bochi-Potrie-S. \cite{BPS} for $\K=\R$, the same arguments also give the result for any local field:
\begin{prop}[{\cite[Lemma 2.5]{BPS}}]\label{weakmorselemma} Let $\rho:\G\to\PGL_d(\K)$ be a projective Anosov representation. Then there exists $\eta_\rho>0$ and $L\in\N$ such that for every geodesic segment $\{\alpha_i\}_0^k$ in $\G$  through $e$ with $|\alpha_0|,|\alpha_k|\geq L$ one has 
$$\angle\Big(U_1\big(\rho(\alpha_k)\big),U_{d-1}\big(\rho(\alpha_0)\big)\Big)>\eta_\rho.$$ 
\end{prop}

Proposition \ref{weakmorselemma} is a key ingredient in the construction of boundary maps:
\begin{prop}[{\cite[Lemma 4.9]{BPS}}]\label{p.bdry}
Let $\rho:\G\to\PGL_d(\K)$ be projective Anosov and $(\alpha_i)_0^\infty\subset\G$ a geodesic ray  based at the identity converging to $x\in\bord\G$ then 
$$\xi^1_\rho(x):=\lim_{i\to\infty}U_1\big(\rho(\alpha_i)\big)\quad \xi^{d-1}_\rho(x):=\lim_{i\to\infty}U_{d-1}\big(\rho(\alpha_i)\big)$$
exist, do not depend on the ray and define continuous $\rho$-equivariant transverse maps $\xi^{1}_\rho:\bord\G\to\P(\K^d)$, $\xi^{d-1}_\rho:\bord\G\to\P\big((\K^d)^*\big).$
Furthermore, there are positive constants $C,\mu$ depending only on $\rho$ such that
$$d\Big(U_1\big(\rho(\alpha_k)\big), \xi^1_\rho(x)\Big)\leq C e^{-\mu k}$$
\end{prop}

The following Lemma from concerning properties of boundary maps will be precious in Section \ref{s.ccone}:
\begin{lemma}[{Bochi-Potrie-S. \cite[Lemma 3.9]{BPS}}]\label{lemenyun} Let $\rho:\G\to\PGL_d(\K)$ be projective Anosov, then there exist constants $\nu\in(0,1),$ $a_0>0$ and $a_1>0$ such that for every $\g,\eta\in\G$ one has $$d_\G(\g,\eta)\geq \nu(|\g|+|\eta|)-a_0-a_1|\log d\big(U_1(\rho(\g)),U_1(\rho(\eta))\big)| .$$
\end{lemma}


\section{Hausdorff dimension of the limit set and the affinitiy exponent}\label{sec:2}

Generalizing the definition given in the introduction, we define the affinity exponent $h_\rho^\Aff$ of a projective Anosov representation $\rho:\G\to\PGL(V_\K)$ as the critical exponent of the broken Dirichlet series  
\begin{alignat*}{2} & \Phi_\rho^\Aff(s)  = \\ & \sum_{\g\in\G}\left(\frac{\sigma_2}{\sigma_1}\big(\rho(\g)\big)\cdots\frac{\sigma_{p-1}}{\sigma_1}\big(\rho(\g)\big)\right)^{d_\K} \left(\frac{\sigma_p}{\sigma_1}\big(\rho(\g)\big)\right)^{s-d_\K(p-2)}s\in[d_\K(p-2),d_\K(p-1)]
\end{alignat*} 
where the dimension $d_\K$ of $\P^1(\K)$ is 1 unless $\K=\C$ in which case $d_\C=2$.

Recall furthermore that for a metric space $(\L,d)$ and for $s>0$ one defines its \emph{$s$-capacity} as 
$$\cal H^s(\L)=\inf_{\eps}\left\{\sum_{U\in\cal U}\diam U^s: \cal U\textrm{ is a covering of $\L$ with }\sup_{U\in\cal U}\diam U<\eps\right\}$$ 
and that the \emph{Hausdorff dimension} of $\L$ is defined by 
\begin{equation}
	\label{HffDef}\Hff(\L)=\inf\left\{s:\cal H^s(\L)=0\right\}=\sup\left\{s:\cal H^s(\L)=\infty\right\}.
\end{equation}

The goal of the section is to prove the following result:
\begin{thm}\label{t:3.1}
Let $\K$ be a local field. If $\rho:\G\to\PGL(V_\K)$ is $\aa_1$-Anosov then $$\Hff\big(\xi^1_\rho(\partial \G)\big)\leq h_\rho^\Aff.$$
\end{thm}

The proof of Theorem \ref{t:3.1} is elementary and based on the construction of a good cover of the image of the limit map (explicitely constructed in \S~\ref{s.ccone}) which we show, in \S~\ref{s.ellipse} to be contained in ellipses of controlled axis.
\subsection{Coarse Cone types}\label{s.ccone}
In P.-S.-W. \cite[Section 2.3.1]{PSW1} we used cone types at infinity to construct well behaved coverings of the boundary of the group. For the purposes of this paper a coarse version of these sets will be more useful, which we now introduce.

 Recall that a sequence $(\alpha_j)_0^\infty$ is a $(c_0,c_1)$-quasigeodesic if for every pair $j,l$ it holds
$$\frac1{c_0}|j-l|-c_1\leq d_\G(\alpha_j,\alpha_l)\leq c_0|j-l|+c_1.$$
We associate to every element $\gamma$ a \emph{coarse cone type at infinity}, consisting of endpoints at infinity of quasi geodesic rays based at $\g^{-1}$ passing through the identity:   
\begin{alignat*}{2} \cone^{c_0,c_1}_\infty(\g)& =  \\ & \Big\{[(\alpha_j)_0^\infty]\in\bord\G|\,  (\alpha_i)_0^\infty \text{ is a $(c_0,c_1)$-quasigeodesic with } \alpha_0=\g^{-1}, e\in\{\alpha_j\})\Big\}.\end{alignat*}

\begin{figure}[hbt]
	\centering
	\begin{tikzpicture}[scale = 0.6]
	\draw [thick] circle [radius = 3];\draw [fill] circle [radius = 0.03];	
	\draw  circle [radius = 0.78];
	\node [left] at (-0.65,0) {$B_{c_1}(e)$};
	\draw [thick] (0.7,-2.2) -- (0.5,-1.7) -- (0.004,-1.3) -- (0.1,-0.7) -- (-0.2,-0.3) -- (0,0); 
	\node [left] at (0.7,-2.2) {$\g^{-1}$}; \draw [fill] (0.7,-2.2) circle [radius = 0.03]; 
	\draw [thick] (1,0.2) -- (0.1,-0.7);
	\draw [thick] (-0.2,-0.3) -- (-0.3,1);
	\draw [thick] (0,0) -- (-0.3,1) -- (0.2,2) -- (0,3); 
	\draw [thick] (0.2,2) -- (0.4, 2.7) -- (0.7,2.91); 
	\draw [thick] (1,0.2) -- (0.89,1.37);
	\draw [thick] (1.8,0.8) -- (2.7,1.3);
	\draw [thick] (0,0) -- (1,0.2) -- (1.8,0.8) -- (2.846,0.948);
	
	\draw [fill] (0.7,2.91) circle [radius = 0.03];
	\draw [fill] (2.7,1.3) circle [radius = 0.03];
	\draw [fill] (2.12,2.12) circle [radius = 0.03];
	\draw [fill] (0,3) circle [radius = 0.03];
	\draw [fill] (2.846,0.948) circle [radius = 0.03];	
	\draw [thick] (0,0) -- (0.89,1.37) -- (1.2,2) -- (2.12,2.12);
	\draw [thick] (1.2,2) -- (1.4,2.65);
	\draw [fill] (1.4,2.65) circle [radius = 0.03];
	\node [right] at (3,-2) {$\G$};
	\node [right] at (2,3) {$\cone^{c_0,c_1}_\infty(\g)$};
	\end{tikzpicture}
	\caption{The coarse cone type at infinity, the black broken lines are $(c_0,c_1)$-quasigeodesics. All endpoints of geodesic rays from $\g^{-1}$ intersecting the ball $B_{c_1}(e)$ clearly belong to $\cone^{c_0,c_1}_\infty(\g)$}\label{fig:1}
\end{figure}
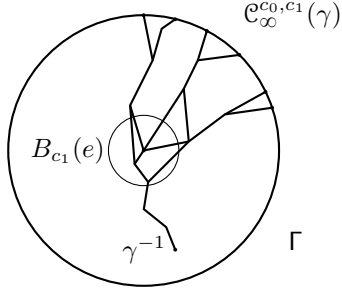

Hyperbolicity of $\G$ lets us understand the overlaps of coarse cone types; this will be crucial in Section \ref{s.PS} to guarantee bounded overlap of suitable covers of the limit set.
\begin{prop}\label{group:finite-intersection} 
Let $\G$ be word-hyperbolic. For every $c_0,c_1$ there exists $C>0$ such that if
 $$\g\cone_\infty^{c_0,c_1}(\g)\cap\eta\cone_\infty^{c_0,c_1}(\eta)\neq\emptyset$$ 
 then 
$$d_\G(\g,\eta)\leq \big||\g|-|\eta|\big|+C.$$
\end{prop}

\begin{proof} 
Assume that $x \in\g\cone_\infty^{c_0,c_1}(\g)\cap\eta\cone_\infty^{c_0,c_1}(\eta)$. Since $\G$ is hyperbolic, by the Morse Lemma, there exists $K>0$ (only depending on $c_0,c_1$ and the hyperbolicity constant of $\G$) such that $\g$ is at distance at most $K$ from a geodesic ray from $e$ to $x.$ The same holds then for $\eta$, and, using the hyperbolicity of $\G$ again, we can assume, up to making the constant $K$ worse (but still depending on $c_0,c_1$ only), that the two rays agree. This implies that there exist $g_0$ and $g_1$ on a geodesic ray from $e$ to $x$ such that $d(\g,g_0)\leq K$ and $d(\eta,g_1)\leq K.$ Since $g_0$ and $g_1$ lie in a geodesic we have $d(g_0,g_1)\leq\big||g_0|-|g_1|\big|$ and thus $$d(\g,\eta)\leq 4K+\big||\g|-|\eta|\big|.$$
\end{proof}
Our next goal is to show that, for an Anosov representation, the intersections of Cartan's basins of attraction $B_{\theta,\alpha}(\rho(\g))$ with the image of the boundary map are contained in the image of a suitably big coarse cone type of $\gamma$. 
Let now $\theta\subset \Pi$ be a subset containing the first root $\sroot_1$. We will denote by $\pi_{\theta,1}:\cal F_\theta(\K^d)\to\P(\K^d)$ the canonical projection. Recall from (\ref{Btheta}) that, for every $\alpha$, we asociate to each $g\in\PGL(V_\K)$ a basin of attraction $B_{\theta,\alpha}(g)\subset\cal F_\theta$.  We will now use Lemma \ref{lemenyun}  to show that, for every $\alpha$, there exist $c_0,c_1$ such that the intersection of a $\theta$-basin of attraction $B_{\theta,\alpha}(\rho(\gamma))$ with the image of the boundary map is contained in a $(c_0,c_1)$-coarse cone type.  

\begin{prop}\label{basin-conetypes}
Let $\rho:\G\to\PGL(V_\K)$ be projective Anosov and consider $\alpha>0.$ There exist $c_0,c_1$ only depending on $\alpha$ and $\rho$ such that for every $\theta\subset \Pi$ containing $\sroot_1$, and every $\g\in\G$
$$(\xi^1)^{-1}\left( \pi_{\theta,1}\left(B_{\theta,\alpha}(\rho(\gamma))\right)\right)\subset \bordcone^{c_0,c_1}(\gamma).$$
\end{prop}
\begin{proof}
It is enough to show that if $\xi^1(x) \in \pi_{\theta,1} (B_{\theta,\alpha}(\rho(\gamma)))$ and $|\gamma|$ is big enough, then there is a quasi-geodesic ray from  $\g^{-1}$  to $x$ that passes through the identity whose constants only depend on $\alpha$ and $\rho$. Consider a quasigeodesic ray $\{\alpha_j\}$ converging to $x$, and fix $1>\alpha'>\alpha$. Since, by assumption, $\xi^1(x)\in B_{\sroot_1,\alpha}(\rho(\g))$, we can find a constant $L$ depending on $\rho$ only, such that for every $j>L$ it holds $U_1(\rho(\alpha_i))\in B_{\aa_1,\alpha'}(\rho(\g))$ (the uniformity of $L$ follows from the last statement in Proposition \ref{p.bdry}). By definition we have $\angle(U_1(\rho(\alpha_j)), U_{d-1}(\rho(\gamma^{-1})))>\alpha'$, and thus, in particular, $d(U_1(\rho(\alpha_j)), U_{1}(\rho(\gamma^{-1}))>\alpha'$. Let now $(\alpha_j)_{i=0}^{-|\gamma|_S}$ be a geodesic segment with $\alpha_0=e$, $\alpha_{-|\g|_S}=\g$. Up to further enlarging $\alpha'$ and $L$ (depending on the representation only) we have also that $d(U_1(\rho(\alpha_{-L})), U_{1}(\rho(\alpha_L))>\alpha'$. Lemma \ref{lemenyun} implies that the  sequence $(\alpha_j)_{i=-|\gamma|_S}^\infty$ obtained as concatenation of the geodesic between $\gamma^{-1}$ and the identity and the ray from the identity to $x$ is a quasi geodesic ray, thus the result.
\end{proof}

\begin{cor}\label{finite-intersection}
Let $\rho:\G\to\PGL(V_\K)$ be projective Anosov and consider $\alpha>0.$ There exists $C$ only depending on $\alpha$ and $\rho$ such that for every $\theta\subset \Pi$ containing $\sroot_1$, if
$$\xi^1(\bord\G)\cap \pi_{\theta,1}\Big(\rho(\gamma)\cdot B_{\theta,\alpha}(\rho(\gamma))\cap \rho(\eta)\cdot B_{\theta,\alpha}(\rho(\eta))\Big)\neq \emptyset$$
then 
$$d(\g,\eta)\leq \big||\g|-|\eta|\big|+C.$$
\end{cor}
\begin{proof}
This follows immediately combining Proposition \ref{basin-conetypes} and Proposition \ref{group:finite-intersection}.
\end{proof}

In particular we can use basins of attraction to construct coverings of the image of the boundary map with bounded overlap:

\begin{prop}[{Cfr. P.-S.-W. {\cite[Lemma 2.21]{PSW1}}}]\label{SullivanShadows} 
Let $\rho:\G\to\PGL(V_\K)$ be projective Anosov.
There exists $\alpha$ small enough so that, for every $T>0$,  the family of open sets 
$$\mathcal U_T:=\{\rho(\g)\cdot B_{\sroot_1,\alpha}(\rho(\gamma)):|\g|= T\}$$ defines an open covering of $\xi^1(\bord\G).$  Furthermore there exists a constant $C$ depending on $\alpha$ (and $\rho$) such that for every $x\in\bord\G$ and every $T$, $\xi(x)$ is contained in at most $C$ elements of $\mathcal U_T$.
\end{prop}
\begin{proof}
Let $x\in\bord\G$, let $\{\gamma_j\}$ be a geodesic ray based at the identity representing $x$.  Propositions \ref{weakmorselemma} and \ref{p.bdry} guarantee that there exists $\alpha=\alpha_\rho$ such that 
$$\angle (\rho(\gamma_T^{-1})\xi^1(x), U_{d-1}(\rho(\gamma_T^{-1}))>\alpha,$$
therefore $\xi^1(x)\in \rho(\gamma_T)B_{\sroot_1,\alpha}(\rho(\gamma_T))$.
The second statement is a direct consequence of Corollary \ref{finite-intersection}.
\end{proof} 

\subsection{Ellipses}\label{s.ellipse}
 The purpose of this section is to prove that for a projective Anosov representation, the set $\rho(\g)\cdot B_{\sroot_1,\alpha}(\rho(\gamma))$
 is coarsely contained in an ellipsoid with axes of size $$\frac{\sigma_2}{\sigma_1}(\rho(\g)),\ldots,\frac{\sigma_{d}}{\sigma_1}(\rho(\g)).$$
\begin{defi}
Let $V$ be a $d$-dimensional $\K$-vector space with $\K$-norm $\|\cdot\|$. Let $$u_1\oplus \cdots\oplus u_d$$ be a $\K$-orthogonal decomposition and let $v=\sum v_j u_j$ be the associated decomposition of $v\in V$, for suitable $v_j\in\K$. Choose positive real numbers $a_2\geq\ldots a_d\geq 1$. If $\K$ is Archimedean, an \emph{ellipsoid} about $\K u_1$ is the projectivisation of 
$$\{v\in V: |v_1|^2\geq\sum_2^d (a_j|v_j|)^2\}$$ 
for some $a_i>0.$
If, instead, $\K$ is non-Archimedean, an \emph{ellipsoid} about $\K u_1$ is the projectivisation of $$\left\{v\in V: |v_1|\geq\max_{2\leq i\leq d} (a_j|v_j|)\right\}$$
\end{defi}
The vector spaces $u_1\oplus u_j$ are \emph{the axes} of the ellipsoid and the \emph{size} of the axis $u_1\oplus u_j$ is $1/a_j.$  
We need the following covering lemma.

\begin{lemma}\label{coveringLemma} Let $E$ be an ellipsoid with axis of size $1\geq\beta_2\geq\ldots \geq \beta_d$. For every $p\in\lb2,d\rb,$ $E$ can be covered by 
$$2^{2p}\left(\frac{\beta_2\cdots \beta_{p-1}}{\beta_ p^{p-2}}\right)^{d_\K}$$ 
balls of radius $\sqrt d\beta_p$.
\end{lemma}

\begin{proof} We consider the affine chart of $\P(V)$ corresponding to $u_1=1$. The ellipsoid $E$ is contained in the product of the balls $\{|v_i|\leq \beta_i\}\subset\K$ (it  agrees with such product if $\K $ is non-Archimedean). If $\K$ is Archimedean, the ball $\{|v_j|\leq \beta_j\}$ is contained in the union of $\left\lceil\frac{\beta_j}{\beta_ p}\right\rceil^{d_\K}$ balls of radius $\beta_p$. Since the product of $d$ balls of radius $\beta_p$ is contained in a ball of radius $\sqrt d\beta_p$ we obtain that $E$ can be covered by 
$$\left\lceil\frac{\beta_2}{\beta_ p}\right\rceil^{d_\K}\cdots \left\lceil\frac{\beta_{p-1}}{\beta_ p}\right\rceil^{d_\K}$$ 
balls of radius $\sqrt d\beta_p.$

If, instead, $\K$ is non-Archimedean, the ball $\{|v_j|\leq \beta_j\}$ can be decomposed in $q^{\left\lfloor\log_q(\frac{\beta_j}{\beta_ p})\right\rfloor}$ balls of radius $\beta_p$, and hence $E$ can be covered with 
$$q^{\left\lfloor\log_q\left(\frac{\beta_2}{\beta_ p}\right)\right\rfloor}\ldots q^{\left\lfloor\log_q\left(\frac{\beta_{p-1}}{\beta_ p}\right)\right\rfloor}$$
balls of radius $\beta_p$.\end{proof}

\begin{prop}\label{ellipsesLemma} Consider $\alpha>0$. For every $g\in\PGL(V_\K)$ one has that the image of the corresponding Cartan's basin of attraction $g\cdot B_{\sroot_1,\alpha}(g)$ is contained in the ellipsoid about $U_1(g)$ with axes $u_1(g)\oplus u_j(g)$ of size $$\frac1{\sin\alpha}\frac{\sigma_j}{\sigma_1}(g).$$
\end{prop}

\begin{proof} 
Assume first that $\K$ is Archimedean. By definition of $B_{\sroot_1,\alpha}(g)$, for every $v\in \K^d$ with $\K\cdot v\in B_{\sroot_1,\alpha}(g)$ one has 
$$|v_1|^2\geq (\sin\alpha)^2\sum_1^d |v_j|^2,$$ where $(v_1,\cdots,v_d) $ are the coefficients in the decomposition of $v$ with respect to the orthogonal splitting $V=\bigoplus g^{-1}u_j(g).$

Since the coefficients  $w_j$ of  $gv$ in the decomposition  induced by the orthogonal decomposition $V=\bigoplus u_j(g)$  satisfy $|w_j|=\sigma_j(g)|v_j|$, one has 
\begin{alignat*}{2}
|w_1|^2=\sigma_1(g)^2|v_1|^2 & \geq\sigma_1(g)^2(\sin \alpha)^2\sum_{j=2}^d|v_j|^2 \nonumber\\ 
& =\sigma_1(g)^2(\sin\alpha)^2\sum_{j=2}^d\frac1{\sigma_j(g)^2}|w_j|^2.
\end{alignat*} 
One concludes that $gv$ lies on the corresponding ellipsoid. The non-Archimedean case follows analogously.\end{proof}

\subsection{The lower bound on the affinity exponent}
We now have all the ingredients needed to prove Theorem \ref{t:3.1}:
\begin{proof}

For each $T>0$ denote by $\cal U_T$ the covering of $\xi^1(\bord\G)$ given by Proposition \ref{SullivanShadows}. By definition, $U=U_\g\in\cal U_T$ is of the form $\rho(\g)\cdot B_{\sroot_1,\alpha}(\rho(\gamma))$ for some $\g$ satisfying $|\g|= T.$ 
Proposition \ref{ellipsesLemma} applied to $\rho(\g)$ implies that $\rho(\g)\cdot B_{\sroot_1,\alpha}(\rho(\gamma))$ is contained in an ellipsoid about $\K u_1(\rho(\g))$ of axes with sizes $$\frac1{\sin\alpha}\frac{\sigma_2}{\sigma_1}(\rho(\g)),\ldots,\frac1{\sin\alpha}\frac{\sigma_{d}}{\sigma_1}(\rho(\g)).$$
Furthermore, since $\rho$ is Anosov, we deduce from Lemma \ref{l.basin} that  $\sup_{U\in \cal U_T}\diam U$ is arbitrarily small as $T$ goes to infinity. Recall that the $s$-capacity
$\cal H^s$ was defined by Equation (\ref{HffDef}).
Applying the covering Lemma \ref{coveringLemma} to these ellipses and any $p\in\lb2,d\rb,$ we obtain
$$\cal H^s(\xi(\bord\G))\leq 2^{2p} \left(\frac {\sqrt d}{\sin\alpha}\right)^s\inf_T\sum_{\g:|\g|\geq T}\left(\frac{\sigma_2}{\sigma_1}\big(\rho(\g)\big)\cdots\frac{\sigma_{p-1}}{\sigma_1}\big(\rho(\g)\big)\right)^{d_\K} \left(\frac{\sigma_p}{\sigma_1}\big(\rho(\g)\big)\right)^{s-d_\K(p-2)}.$$

By definition, the affinity exponent $h^\Aff_\rho$ is such that for all $s>h^\Aff_\rho$ the broken Dirichlet series 
$$\sum_{\g\in\G}\left(\frac{\sigma_2}{\sigma_1}\big(\rho(\g)\big)\cdots\frac{\sigma_{p-1}}{\sigma_1}\big(\rho(\g)\big)\right)^{d_\K}
 \left(\frac{\sigma_p}{\sigma_1}\big(\rho(\g)\big)\right)^{s-d_\K(p-2)}:s\in[d_\K(p-2),d_\K(p-1)]$$ 
is convergent and thus for all $s>h^\Aff_\rho$ we have
$$ 2^p\left(\frac {\sqrt d}{\sin\alpha}\right)^s\inf_T\sum_{\g:|\g|\geq T}
\left(\frac{\sigma_2}{\sigma_1}\big(\rho(\g)\big)\cdots\frac{\sigma_{p-1}}{\sigma_1}\big(\rho(\g)\big)\right)^{d_\K} \left(\frac{\sigma_p}{\sigma_1}\big(\rho(\g)\big)\right)^{s-d_\K(p-2)}=0.$$ 
As a result we conclude that for all $s>h^\Aff_\rho$ the $s$-capacity $\cal H^s\big(\xi(\bord\G)\big)$ vanishes, hence $$h^\Aff_\rho\geq\Hff\big(\xi(\bord\G)\big).$$ This completes the proof. 
\end{proof}
The following generalization of Corollary \ref{cor:1.3} is also immediate:
\begin{cor}\label{cor:1.3b}
Let $\rho:\G\to\PGL(V_\K)$ be projective Anosov. If $\Hff\big(\xi(\partial\G)\big)\geq p d_\K$, then 
$$\Hff\big(\xi(\partial\G)\big)\leq p h_\rho(\jac^u_{p}).$$
\end{cor}
\begin{proof}
Observe that for every $s\in [d_\K p,d_\K (p+1)]$,  the value of the broken Dirichlet series defining the affinity exponent 
$$\Phi_\rho^\Aff(s)=\sum_{\g\in\G}\left(\frac{\sigma_2}{\sigma_1}\big(\rho(\g)\big)\cdots\frac{\sigma_{p+1}}{\sigma_1}\big(\rho(\g)\big)\right)^{d_\K} 
\left(\frac{\sigma_{p+2}}{\sigma_1}\big(\rho(\g)\big)\right)^{s-d_\K p}$$
is smaller than or equal to the value of the series associated to the $p$-th unstable Jacobian divided by $p$:
$$\Phi_\rho^{\frac 1 { p}\jac^u_{p}}(s)=\sum_{\g\in\G}\left(\frac{\sigma_2}{\sigma_1}\big(\rho(\g)\big)\cdots\frac{\sigma_{p+1}}{\sigma_1}\big(\rho(\g)\big)\right)^{\frac{s}{ p}}.$$ 
Indeed,
\begin{alignat*}{2} 
\left(\frac{\sigma_{p+2}}{\sigma_1}\big(\rho(\g)\big)\right)^{s-d_\K p} & =\left(\frac{\sigma_{p+2}}{\sigma_1}\big(\rho(\g)\big)\right)^{ p(\frac s{p}-d_\K)}\\ &\leq \left(\frac{\sigma_{2}}{\sigma_1}\big(\rho(\g)\big)\right)^{(\frac s{p}-d_\K)}\cdots\left(\frac{\sigma_{p+1}}{\sigma_1}\big(\rho(\g)\big)\right)^{(\frac s{p}-d_\K)}.
\end{alignat*}

\noindent
As a result,  if $d_\K p\leq h^\Aff_\rho\leq d_\K(p+1)$, then $p h_\rho(\jac^u_{p})\geq  h^\Aff_\rho$.

The result follows as, for all $k\in\lb1,d-1\rb$  and $v\in\EE^+$ one has that 
$$\frac{\jac^u_{k-1}(v)}{k-1}\leq\frac{\jac^u_{k}(v)}{k},$$
 which implies $kh_\rho(\jac^u_{k})\leq (k-1)h_\rho(\jac^u_{k-1}).$
\end{proof}

\section{Semi-simple algebraic groups}\label{semi-simplegroups}

Let $\sf G$ be a connected semi-simple $\K$-group, $\sf G_\K$ the group of its $\K$-points, $\sf A$ a maximal $\K$-split torus and $\mathbf{X}(\sf A)$ the group of its $\K^*$-characters. Consider the real vector space $\sf E^*=\mathbf{X}(\sf A)\otimes_\Z\R$ and $\sf E$ its dual.  For every $\chi\in\mathbf X(\sf A),$ we denote by $\chi^\omega$ the corresponding linear form on $\sf E.$

\subsection{Restricted roots and parabolic groups}

Let $\Sigma$ be the set of restricted roots of $\sf A$ in $\frak g,$ the set $\Sigma^\omega$ is a root system of $\sf E^*.$ Let $\Sigma^+$ be a system of positive roots and $\Pi$ the associated subset of simple roots. Let $\sf E^+$ be the Weyl chamber determined by the positive roots $(\Sigma^\omega)^+.$

Let $\Weyl$ be the Weyl group of $\Sigma,$ it is isomorphic to the quotient of the normalizer $N_{\sf G_\K}(\sf A_\K)$ of $\sf A_\K$ in $\sf G_\K$ by its centralizer $Z_{\sf G_\K}(\sf A_\K).$ Let $\ii:\sf E\to\sf E$ be the opposition involution: if $u:\sf E\to\sf E$ is the unique element in the Weyl group with $u(\sf E^+)=-\sf E^+$ then $\ii=-u.$

A subset $\t\subset\Pi$ determines a pair of opposite parabolic subgroups $\sf P_\t$ and $\check{\sf P}_\t$ whose Lie algebras are defined by $$\frak p_\t=\bigoplus_{\sroot\in\Sigma^+\cup\{0\}}\frak g_{\aa} \oplus\bigoplus_{\sroot\in\<\Pi-\t\>}\frak g_{-\aa},$$ and $$\check{\frak p}_\t=\bigoplus_{\sroot\in\Sigma^+\cup\{0\}}\frak g_{-\aa} \oplus\bigoplus_{\sroot\in\<\Pi-\t\>}\frak g_{\aa}.$$ The group $\check{\sf P}_\t$ is conjugated to the parabolic group $\sf P_{\ii\t}.$ Let  $$\frak l_\t=\frak p_\t\cap\check{\frak p}_\t$$ be the Lie algebra of the associated  Levi group.

The $\K$-\emph{flag space} associated to $\t$ is $\cal F_\t(\sf G_\K)=\sf G_\K/\sf P_{\t,\K},$ the $\sf G_\K$ orbit of the pair $([\sf P_{\t,\K}],[\check{\sf P}_{\t,\K}])$ is the unique open orbit for the action of $\sf G_\K$ in the product $\cal F_\t(\sf G_\K)\times\cal F_{\ii\t}(\sf G_\K).$ This orbit is denoted by $\posgen_\t(\sf G_\K).$

For $y\in\cal F_{\ii\t}(\sf G_\K)$ denote by \begin{equation}\ann(y)=\{x\in\cal F_\t(\sf G_\K):(x,y)\notin\cal F_\t(\sf G_\K)^{(2)}\}\end{equation} the closed submanifold of flags in $\cal F_\t(\sf G_\K)$ that are not transverse to $y.$

Denote by $(\cdot,\cdot)$ a $\Weyl$-invariant inner product on $\sf E,$ $(\cdot,\cdot)$ the induced inner product on $\sf E^*$ and define $$\<\chi,\psi\>=\frac{2(\chi,\psi) }{(\psi,\psi)}$$ and let $\{\om_\aa\}_{\aa\in\Pi}$ be the \emph{dual basis} of $\Pi,$ i.e. $\<\om_\aa,\sf b\>=\delta_{\aa\sf b}.$ The linear form $\om_\aa$ is \emph{the fundamental weight} associated to $\aa.$

\subsection{Cartan decomposition}

Let $\nu:\sf A_\K\to\sf E$ be defined, for $z\in\sf A_\K,$ as the unique vector in $\sf E$ such that for every $\chi\in\mathbf X(\sf A)$ one has $$\chi^\omega(\nu(z))=\log|\chi(z)|.$$ Denote by $\sf A_\K^+=\nu^{-1}(\sf E^+).$

Let $K\subset\sf G_\K$ be a compact group that contains a representative for every element of the Weyl group $\Weyl.$ This is to say, such that the normaliser $N_{\sf G_\K}(\sf A_\K)$ verifies $N_{\sf G_\K}(\sf A_\K)=(N_{\sf G_\K}(\sf A_\K)\cap K)\sf A_\K.$ One has $\sf G_\K=K\sf A_\K^+K$ and if $z,w\in \sf A_\K^+$ are such that $z\in KwK$ then $\nu(z)=\nu(w).$ There exists  thus a function 
$$\cartan:\sf G_\K\to \sf E^+$$ such that for every $g_1,g_2\in \sf G_\K$ one has that $g_1\in Kg_2 K$ if and only if $\cartan(g_1)=\cartan(g_2).$ It is called the \emph{Cartan projection} of $\sf G_\K.$

In  the case of $G_\K=\PGL(V_\K)$ this is nothing but the ordered list of semihomotecy ratioes   defined in Section \ref{ss.dom_sing}.

\subsection{Representations of $\sf G_\K$}\label{representaciones}

Let $\L:\sf G\to\PGL(V)$ be a finite dimensional rational\footnote{i.e. a rational map between algebraic varieties.} irreducible representation and denote by $\phi_\L:\frak g\to\frak{sl}(V)$ the Lie algebra homomorphism associated to $\L.$ Then the \emph{weight space} associated to $\chi\in\mathbf X(\sf A)$ is the vector space $$V_\chi=\{v\in V:\phi_\L(a) v=\chi(a) v\ \forall a\in\sf A_\K\}$$ and if $V_\chi\neq0$ then we say that $\chi^\omega\in\sf E^*$ is a \emph{restricted weight} of $\Lambda.$ Theorem 7.2 of Tits \cite{Tits} states that the set of weights has a unique maximal element with respect to the order $\chi\geq\psi$ if $\chi-\psi$ is positive on $\sf E^+.$ This is called \emph{the highest weight} of $\L$ and denoted by $\chi_\L.$

\begin{defi}Let $\t_\L$ be the set of simple roots $\aa\in\Pi$ such that $\chi_\L-\aa$ is still a weight of $\L.$ \end{defi}

\begin{obs}\label{igual0} The subset $\t_\L$ is the subset of simple roots such that the following holds: Consider $\aa\in\E^+,$ $n\in\frak g_{-\aa}$ and $v\in\chi_\L,$ then $\phi_\L(n)v=0$ if and only if $\aa\in\<\Pi-\t_\L\>.$
\end{obs}

\begin{defi}\label{d.normL}
We denote by $\|\,\|_\L$ a good norm  on $V$ invariant under $\L K$ and such that $\L\sf A_\K$ consists on semi-homotecies, if $\K$ is Archimedean the existence of such a norm is classical, if $\K$ is non-Archimedean then this is the content of Quint \cite[Th\'eor\`eme 6.1]{Quint-localFields}. 
\end{defi}
For every $g\in \sf G_\K$ one has 
\begin{equation}\label{eq:normayrep}
\log\|\L g\|_\L=\chi_\L(\cartan(g)).
\end{equation} 
If $g=k_gz_g l_g$ with $k,l\in K$ and $z_g\in \sf A_\K^+$ then for all $v\in \L(l_g^{-1})V_{\chi_\L}$ one has $\|\L g(v)\|_\L=\|\L g\|_\L\|v\|_\L.$

Denote by $W_{\chi_\L}$ the $\L\sf A_\K$-invariant complement of $V_{\chi_\L}.$ Note that the stabilizer in $\sf G_\K$ of $W_{\chi_\L}$ is $\wk{\sf P}_{\t,\K},$ and thus one has a map of flag spaces \begin{equation}\label{maps}(\xi_\L,\xi^*_\L):\cal F_{\t_\L}^{(2)}(\sf G_\K)\to \Gr_{\dim V_{\chi_\L}}^{(2)}(V).\end{equation} This is a proper embedding which is an homeomorphism onto its image. Here $\Gr_{\dim V_{\chi_\L}}^{(2)}(V)$ is the open $\PGL (V_\K)$-orbit in the product of the Grassmannian of $(\dim V_{\chi_\L})$-dimensional subspaces and the Grassmannian of $(\dim V-\dim V_{\chi_\L})$-dimensional subspaces.

One has the following proposition by Tits (see also Humphreys \cite[Chapter XI]{LAG}).

\begin{prop}[Tits \cite{Tits}]\label{prop:titss} 
For each $\aa\in\Pi$ there exists a finite dimensional rational irreducible representation $\L_\aa:\sf G\to\PSL(V_\aa),$ such that $\chi_{\L_\aa}$ is an integer multiple of the fundamental weight $\om_\aa$ and $\dim V_{\chi_{\L_\aa}}=1.$ All other weights of $\L_\aa$ are of the form 
$$\chi_\aa-\aa-\sum_{\sf b\in\Pi}n_\sf b\sf b,$$ 
where $n_\sf b\in\N.$
\end{prop}

We will fix from now on such a set of representations and call them, for each $\aa\in\Pi,$ the \emph{Tits representation associated to $\aa$}.

\subsection{The center of the Levi group $\sf P_{\t,\K}\cap\check{\sf P}_{\t,\K}$}

We now consider the vector subspace 
$$\sf E_\t=\bigcap_{\aa\in\Pi-\t}\ker\aa^\omega$$ 
together with the unique projection $\pi_\t:\sf E\to\sf E_\t,$ that is invariant under the subgroup $\Weyl_\t$ of the Weyl group spanned by reflections associated to roots in $\Pi-\t$: $$\Weyl_\t=\{w\in \Weyl:w(v)=v\, \forall v\in\sf E_\t\}.$$

The dual space $(\sf E_\t)^*$ is canonically the subspace of $\sf E^*$ of $\pi_\t$-invariant linear forms and it is spanned by the fundamental weights of roots in $\t$ 
$$(\sf E_\t)^*=\{\varphi\in\sf E^*:\varphi\circ\pi_\t=\varphi\}=\<\omega_\sroot:\sroot\in\t\>.$$ 

Since $\pi_\t^2=\pi_\t,$ the pre-composition with $\pi_\t$ induces a projection $\sf E^*\to(\sf E_\t)^*$ denoted by $$\varphi\mapsto\varphi^\t:=\varphi\circ\pi_\t.$$

The following examples will be relevant in Section \ref{S.C^1} and \ref{sec:5} respectively:
\begin{ex}\label{a_p^t}Let $\sf G_\K=\PGL(V_\K),$ consider $p\in\lb2,d-2\rb$ and let $\t=\{\sroot_1,\sroot_p,\sroot_{d-1}\},$ so that \begin{alignat*}{3}\sf E_\t=\{(a_1,\ldots,a_d)\in\sf E:\, & a_2&  = &  \cdots=a_p \textrm{ and } \\ & a_{p+1}& =& \cdots=a_{d-1}\}
\end{alignat*} is three dimensional. Using the fact that the fundamental weights $\omega_i$ (for $i=1,p,d-1$) belong to $(\sf E_\t)^*$ one checks that the projection is \begin{alignat*}{1}
\eps_1(\pi_\t(a)) & = a_1,\\
\eps_i(\pi_\t(a)) &= \frac{a_2+\cdots+a_p}{p-1}=\frac{\omega_p-\omega_1}{p-1}(a)\textrm{ for every }i\in\lb2,p\rb,\\
\eps_i(\pi_\t(a)) & = \frac{a_{p+1}+\cdots+a_{d-1}}{d-p-1}=\frac{\omega_{d-1}-\omega_p}{d-p-1}(a)\textrm{ for every }i\in\lb p+1,d-1\rb,\\
\eps_d(\pi_\t(a)) & = a_d.
\end{alignat*}

\noindent
One has then that $$\sroot_p^\t=\frac{\omega_p-\omega_1}{p-1}-\frac{\omega_{d-1}-\omega_p}{d-p-1}$$  and that $\sroot^\t_p|_{\sf E^+-\{0\}}\geq\sroot_p|_{\sf E^+-\{0\}}.$
\end{ex}

\begin{ex}\label{so(p,q)} Consider the group $\SO(p,q)$ of transformations in $\PSL_{p+q}(\R)$ preserving a signature $(p,q)$ bilinear form with $p<q.$ One has that $$\sf E=\{(a_1,\ldots,a_p):a_i\in\R\}$$ equipped with the root system $$\Sigma^\omega=\{\eps_i:i\in\lb1,p\rb\}\cup\{a\mapsto a_i-a_j:i,j\in\lb1,p\rb\}.$$ A Weyl chamber can be chosen as $$\sf E^+=\{a\in\sf E:a_i\geq a_{i+1}\, \forall i\in\lb1,p-1\rb\textrm{ and }a_p\geq0\}$$ and the associated set of simple roots $$\Pi=\{\aa_i:i\in\lb1,p-1\rb\}\cup\{\eps_p\}.$$ Consider then $\t=\{\aa_i:i\in\lb1,p-1\rb\},$ so that $\sf E_\t=\ker \eps_p$ and thus $\aa_i\in(\sf E_\t)^*$ for $i\in\lb1,p-2\rb.$ Moreover, $$\aa_{p-1}^\t=\eps_{p-1}$$ and one has that $\aa_{p-1}^\t|_{\sf E^+-\{0\}}\geq\aa_{p-1}|_{\sf E^+-\{0\}}.$  
\end{ex}

\subsection{Gromov product}

Recall from S. \cite{orbitalcounting} that the \emph{Gromov product}\footnote{This is the negative of the  defined in S. \cite{orbitalcounting}.}  based at $K$ is the map 
$$(\cdot|\cdot)_K:\posgen_\t(\sf G_\K)\to\sf E_\t$$ 
defined to be the unique vector $(x|y)_K\in\sf E_\t$ such that 
$$\chi_\aa((x|y)_K)=-\log\sin\angle_{\|\;\|_{\L_\aa}}(\xi_{\L_\aa}x,\xi^*_{\L_\aa}y)$$ 
for all $\aa\in\t,$ where $\chi_\aa$ is the fundamental weight associated to the Tits representation $\L_\aa$ of $\aa.$ Note that 
\begin{equation}\label{minangulo}
\max_{\aa\in\t}\chi_\aa((x|y)_K)=\max_{\aa\in\t}|\chi_\aa((x|y)_K)|=-\log\min_{\aa\in\t}\sin\angle_{\|\,\|_{\L_\aa}}(\xi_{\L_\aa}x,\xi^*_{\L_\aa}y).
\end{equation} 
On has the following remark from Bochi-Potrie-S. \cite{BPS}.

\begin{obs}[{\cite[Remark 8.11]{BPS}}]\label{GromovyRep} Let $\L:\sf G\to\PGL(V)$ be a finite dimensional rational irreducible representation, if $(x,y)\in\posgen_{\t_\L}(\sf G_\K)$ then $$(\xi_\L x |\xi_\L^* y)_{\|\,\|_\L}=\chi_\L((x|y)_K).$$
where $\|\,\|_\L$ denotes the (stabilizer of the) inner product on $V$ such that $\L K$ is orthogonal (see Definition \ref{d.normL}).\end{obs}

\subsection{Iwasawa cocycle and its relation to representations of $\sf G$}
Another important decomposition of Lie groups that will play a role in our work is the Iwasawa decomposition:
$$\sf G_\K=K\sf A_\K\sf U_{\Pi,\K},$$
 where $\sf P_{\Pi,\K}$ is the minimal parabolic subgroup, and $\sf U_{\Pi,\K}$ is its unipotent radical. For general local field $\K$ the decomposition of an element is not necessarily unique, but if $z_1,z_2\in \sf A_\K$ are such that $z_1\in Kz_2\sf U_{\Pi,\K}$, then $\nu(z_1)=\nu(z_2).$ 

Quint used the Iwasawa decomposition to define the Iwasawa cocycle 
$$\bus_\Pi(g,x)=\nu(z)$$
where $x=k[\sf P_{\t,\K}]\in\cal F_\t(\sf G_\K)$ with $k\in K$,  $g\in\sf G_\K$ and $gk$ has Iwasawa decomposition $gk=lzu.$

Quint \cite{quint1} proves the following lemma.

\begin{lemma}[{Quint \cite[Lemmas 6.1 and 6.2]{quint1}}] The map $p_\t\circ\bus_\Pi$ factors trough a map $\bus_\t:\sf G_\K\times\cal F_\t(\sf G_\K)\to \sf E_\t.$ The map $\bus_\t$ verifies the cocycle relation: for every $g,h\in \sf G_\K$ and $x\in\cal F_{\t,\K}(\sf G_\K)$ one has $$\bus_\t(gh,x)=\bus_\t(g,hx)+\bus_\t(h,x).$$
\end{lemma}

One also has the following behavior of $\bus_\t$ under the representations of $\sf G.$

\begin{lemma}[{Quint \cite[Lemma 6.4]{quint1}}]\label{Iwasawa=norm} Let $\L:\sf G\to\PGL(V)$ be a proximal irreducible representation, then for every $x\in\cal F_{\t_\L}(\sf G_\K)$ and $g\in \sf G_\K$ one has $$\chi_\L\big(\bus_{\t_\L}(g,x)\big)=\log\frac{\|\L(g)v\|_\L}{\|v\|_\L},$$ where $v\in\xi_\L(x)-\{0\}.$
\end{lemma}

\subsection{Cartan attractors and Cartan's attracting basins}\label{s.Basin} Consider $g\in\sf G_\K$ and let $g=k_gz_gl_g$ be a Cartan decomposition. Given $\t\subset\Pi,$ the \emph{Cartan attractor of $g$ in $\cal F_\t(\sf G_\K)$} is defined by $$U_\t(g)=U^K_\t(g)=k_g[\sf P_{\t,\K}] ,$$ and the \emph{Cartan basin of $g$} is defined, for $\alpha>0,$ by $$B_{\t,\alpha}(g)= \{x\in\cal F_\t(\sf G_\K): \big(x|U_{\ii\t}(g^{-1})\big)_K<\alpha\}.$$

\begin{obs} If $\L:\sf G\to\PGL(V)$ is a rational irreducible representation with $\t_\L\subset\t$ then $$\xi_\L(U_\t(g))=U_{\dim V_{\chi_\L}}^{\|\,\|_\L}\big(\L(g)\big).$$
\end{obs}

Notice that the flag $U_\t(g)$ is an arbitrary choice of a ``most expanding'' flag of type $\t$ for $g,$ however, it is clear from the definition that given $\alpha>0$ there exists a constant $K_\alpha$ such that if $y\in\cal F_{\t}(\sf G_\K)$ belongs to $B_{\t,\alpha}(g)$ then for all $\aa\in\t$ one has
\begin{equation}\label{comparision-shadow}
|\chi_\aa\big(\cartan(g)-\bus_\t(g,y)\big)|\leq K_\alpha.
\end{equation}

\subsection{The $\PSL_d(\K)$ case} Given a good norm $\tau$ on $\K^d,$ and considering the exterior power representations of $\PSL_d(\K),$ one sees that  Lemma \ref{Iwasawa=norm} provides the following computation for the Iwasawa cocycle $\bus:\PSL_d(\K)\times\cal F(\K^d)\to\sf E$ associated to a maximal compact group stabilizing $\tau.$ For $p\in\lb1,d\rb$ and given $g\in \PSL_d(\K),$ $x\in\cal F(\K^d)$ one has \begin{equation}\label{defBusemann}
\omega_p(\bus(g,x))=\log\frac{\|gv_1\wedge \cdots\wedge gv_p\|}{\|v_1\wedge \cdots\wedge v_p\|}
\end{equation} where $\{v_1,\ldots,v_p\}$ is any basis of the $p$-dimensional space $x^p$ of $x$ and $\|\,\|$ is the norm on $\wedge^p\K^d$ induced by $\tau$.

Notice that, by definition, the number $\omega_p(\bus(g,x))$ only depends on $x^p,$ so in order to simplify notation we will also denote it by $\omega_p(\bus(g,x^p)).$

\section{Patterson-Sullivan measures in non-Anosov directions}\label{Lip}
An interesting quantity associated to a discrete subgroup ${{\Gamma}}<\sf G_\K$ is its critical exponent $h^X_{{\Gamma}}$ which measures the exponential orbit growth rate of orbit points in balls (in the symmetric space of $\sf G_\K$) as the radius grows. The theory of Quint's growth indicator function, which we briefly recall in Section \ref{s.quint} allows to deduce information on $h^X_{{\Gamma}}$ from information on the critical exponent of linear forms $\phi$ on the Weyl chamber $\sf E$, that are often easier to handle with the aid of Patterson-Sullivan measures. When the discrete group ${{\Gamma}}<\sf G_\K$ is the image of an Anosov representation $\rho:\G\to \sf G_\K$, and the form $\phi$ belongs to the dual of the Levi-Anosov subspace $\sf E_{\theta_\rho}$, then the thermodynamical formalism applies (see the Theorem \ref{Q-convex}).

In this section we will, instead, be interested in studying forms $\phi$ that do not belong to  $(\sf E_{\theta_\rho})^*$. Our main result is Theorem \ref{phiInD} in which we show that, provided a representation $\rho$ is Anosov with respect to some root, the existence of Patterson-Sullivan measure in any flag manifold, and thus also in non-Anosov directions $\phi$, have strong implications on the critical exponent of $\phi$.

\subsection{Quint's growth indicator}\label{s.quint} We recall here some definitions from Quint \cite{quint2,quint1}.

Let ${{\Gamma}}\subset\sf G_\K$ be a discrete subgroup, its \emph{Quint growth indicator function \cite{quint2}}  $\cQ_{{\Gamma}}:\sf E^+\to \R_+\cup\{-\infty\}$ is defined as follows. Given a norm $\|\,\|$ on $\sf E$ and an open cone $\scr C\subset\sf E^+$ let $h^{\|\, \|}_{\scr C}$ be the critical exponent of the Dirichlet series 
$$s\mapsto\sum_{g\in{{\Gamma}}:\cartan(g)\in\scr C}e^{-s\|\cartan(g)\|}$$ 
and define $\cQ_{{\Gamma}}:\sf E^+\to\{-\infty\}\cup[0,\infty)$ by $$\cQ_{{\Gamma}}(v)=\|v\|\inf_{v\in\scr C}h^{\|\, \|}_{\scr C},$$ where the infimum is taken over all open cones containing $v.$ One can easily check that $\cQ_{{\Gamma}}$ does not depend on the chosen norm $\|\,\|$ and is 1-positively-homogenous.

Dually one considers the growth on linear forms. The \emph{limit (or Benoist \cite{limite}) cone}  $\cal L_{{\Gamma}}$ of ${{\Gamma}}$ is defined as the limit points of sequences $t_n\cartan(g_n)$ where $(t_n)_{n\in\N}\subset\R_+$ converges to $0$ and $(g_n)_{n\in\N}\subset{{\Gamma}}.$ Denote its dual cone by 
$$(\cal L_{{\Gamma}})^*=\{\varphi\in\sf E^*:\varphi|\cal L_{{\Gamma}}-\{0\}\geq0\},$$
and, for $\varphi\in(\cal L_{{\Gamma}})^*$ let $h_{{\Gamma}}(\varphi)$ be the critical exponent of the Dirichlet series 
$$\sum_{g\in{{\Gamma}}}e^{-s\varphi\big(\cartan(g)\big)},$$ 
that is  
$$h_{{\Gamma}}(\varphi)=\limsup_{t\to\infty}\frac{1}t\log\#\{\g\in\G|\varphi\Big(\cartan\big(\rho(\gamma)\big)\Big)<t\}.$$
One has the following.
\begin{lemma}\label{l.hminmaxh}
It holds
$$h_{{\Gamma}}(\min\{\phi_1,\ldots,\phi_k\})=\max\{h_{{\Gamma}}(\phi_1),\ldots,h_{{\Gamma}}(\phi_k)\}.$$
\end{lemma}	
\begin{proof} One inequality is clear. For the other one, one has
$$	\begin{array}{rl}
		h_{{\Gamma}}(\min\{\phi_1,\ldots,\phi_k\})&\leq\limsup_{t\to\infty}\frac{1}t\log{\displaystyle\sum_{i=1}^k}\#\{\g\in\G|\phi_i(\cartan(\rho(\gamma)))<t\}\\
		&\leq \limsup_{t\to\infty}\frac{1}t\log k\max_i\#\{\g\in\G|\phi_i(\cartan(\rho(\gamma)))<t\}\\
		&= \max\{h_{{\Gamma}}(\phi_1),\ldots,h_{{\Gamma}}(\phi_k)\}
	\end{array}$$
\end{proof}
One can then define the subset 
$$\cD_{{\Gamma}}=\{\varphi\in(\cal L_{{\Gamma}})^*:h_{{\Gamma}}(\varphi)\in(0,1]\}.$$
The next lemma is clear from the definitions, but is very useful in applications:
\begin{lemma}\label{lem} 
If $\phi$ belongs to $\cD_{{\Gamma}}$, then $\phi+\psi\in \cD_{{\Gamma}}$ for every $\psi\in  (\cal L_{{\Gamma}})^*$.
\end{lemma}
The following result from Quint \cite{quint2} allows to deduce information on te critical exponent of various norms in terms of growth of linear functions, that are often easier to compute:

\begin{prop}[{Quint \cite{quint2}}]\label{convexD} One has that $$\cD_{{\Gamma}}=\{\varphi\in\sf E^*:\forall v\in\sf E^+\, \varphi(v)\geq\cQ_{{\Gamma}}(v)\},$$ and thus it is a convex set. Moreover, for any $1$-positively-homogenous function $\Theta:\sf E^+\to\R$ the critical exponent $h_{{\Gamma}}(\Theta)$ of the Dirichlet series 
$$s\mapsto\sum_{g\in{{\Gamma}}}e^{-s\Theta\big(\cartan(g)\big)}$$ 
can be computed as ${\displaystyle h_{{\Gamma}}(\Theta)=\sup_{v\in\sf E^+} \frac{\cQ_{{\Gamma}}(v)}{\Theta(v)}.}$
\end{prop}

The importance of the set $\cD_{{\Gamma}}$ is provided by the following theorem: it is possible to compute the orbit growth rate with respect to various norms studying properties of the set $\cD_{{\Gamma}}$:

\begin{thm}[{Quint \cite{quint2}}]\label{infNorm}
 If the Zariski closure of ${{\Gamma}}$ is semi-simple then $\cQ_{{\Gamma}}$ is concave, consequently for every norm $\|\,\|$ on $\sf E$ one has 
$$h^{\|\,\|}_{{\Gamma}}=\inf\{\|\varphi\|^*:\varphi\in\cD_{{\Gamma}}\}$$ 
where $\|\,\|^*$ is the induced operator norm on $\sf E^*.$
\end{thm}

\begin{obs}\label{o.expSymm} Recall that, if we endow the symmmetric space (or the affine building) $X$ associated to $\sf G_\K$ with a $\sf G_\K$-invaraint Riemannian metric, there exists an Euclidean norm $\|\,\|_X$ on $\sf E$ such that for every $g\in\sf G_\K$ one has $$d_X([K],g[K])=\|\cartan(g)\|_X.$$ Theorem \ref{infNorm} provides then the following formula for the critical exponent of a discrete group with reductive Zariski-closure in the symmetric space $X$:$$h_{{\Gamma}}^X=\inf\{\|\phi\|^*_X:\, \phi\in \cD_{{{\Gamma}}} \}.$$
\end{obs}

The topological boundary $\cal Q_{{\Gamma}}$ of $\cD_{{\Gamma}}$ will be called \emph{Quint's indicator set of} ${{\Gamma}}.$ We will also denote by $$\cal Q_{{{\Gamma}},\t}=\cal Q_{{\Gamma}}\cap(\sf E_\t)^*.$$

Let us record here a useful direct consequence of the convexity of $\cD_{{\Gamma}}.$

\begin{lemma}\label{entropySum}
Let $\phi,\varphi\in(\cal L_{{\Gamma}})^*,$ then 
$$h_{{\Gamma}}(\phi+\varphi)\leq\frac{h_{{\Gamma}}(\phi)h_{{\Gamma}}(\varphi)}{h_{{\Gamma}}(\phi)+h_{{\Gamma}}(\varphi)}.$$\end{lemma}

\noindent
We end this subsection with the following definition from Quint \cite{quint1}.

\begin{defi}Given $\t\subset \Pi$ and $\varphi\in(\sf E_\t)^*$ a \emph{$({{\Gamma}},\varphi)$-Patterson-Sullivan measure} on $\cal F_\t(\sf G_\K)$ is a finite Radon measure $\mu$ such that for every $g\in{{\Gamma}}$ one has $$\frac{\sf d g_*\mu}{\sf d\mu}(x)=e^{-\varphi\big(\bus_\t(g^{-1},x)\big)}.$$ 
\end{defi}

\subsection{Anosov representations with values in $\sf G_\K$}

Let $\G$ be a discrete group and fix $\t\subset\Pi.$ 

\begin{defi}\label{Anosovgral}A representation $\rho:\G\to\sf G_\K$ is $\t$-\emph{Anosov} if there exist constants $c\geq0$ and $\mu>0$ such that for every $\g\in\G$ and $\sroot\in\t$ one has $$\sroot\Big(\cartan\big(\rho(\g)\big)\Big)\geq \mu|\g|-c.$$ 
\end{defi}

If $\rho:\G\to \sf G_\K$ is $\t$-Anosov and  $\Lambda_\aa$ is as in Proposition \ref{prop:titss}, then $\Lambda_\aa\rho:\G\to \sf \PGL(V_\K)$ is projective Anosov. In particular subsection \ref{AnosovSL} applies to arbitrary $\sf G_\K$ and one obtains the following result.

\begin{thm}[Kapovich-Leeb-Porti {\cite{KLP-Morse}}]If $\rho:\G\to\sf G_\K$ is $\t$-Anosov then $\G$ is word-hyperbolic and there exist continuous equivariant maps $\xi^\t_\rho:\bord\G\to\cal F_\t(\sf G_\K)$ and $\xi^{\ii\t}_\rho:\bord\G\to\cal F_{\ii\t}(\sf G_\K)$ such that the product map $(\xi^\t_\rho,\xi^{\ii\t}_\rho):\bord^{(2)}\G\to\posgen_\t(\sf G_\K)$ is transverse.
\end{thm}
We will sometime use the notation introduced in \cite{PSW1} and, if $x\in\bord\G$ is a point, denote by 
$$x^\t_\rho:=\xi^\t_\rho(x)\in\cal F_\t(\sf G_\K)$$
 the image of $x$ via the boundary map. If $\theta=\{\sroot_k\}$ consists of a single root we will also write $\xi^k_\rho$  and $x^k_\rho$ instead of $\xi^{\{\sroot_k\}}_\rho$ and  $x^{\{\sroot_k\}}_\rho$.

If $\t\subset \Pi$ contains the root $\aa$, we denote by $\pi_\aa:\cal F_\theta(\sf G_\K)\to\cal F_\aa(\sf G_\K)$ the natural projection. It is  easy to deduce from Corollary \ref{finite-intersection} the following more general statement:
\begin{cor}\label{finite-intersectionG}
Let $\rho:\G\to\sf G_\K$ be $\aa$-Anosov and consider $\alpha>0.$ There exists $C$ only depending on $\alpha$ and $\rho$ such that for every $\theta\subset \Pi$ containing $\aa$, if
$$\xi_\rho^\aa(\bord\G)\cap \pi_\aa\Big(\rho(\gamma)\cdot B_{\theta,\alpha}(\rho(\gamma))\cap \rho(\eta)\cdot B_{\theta,\alpha}(\rho(\eta))\Big)\neq \emptyset$$
then 
$$d(\g,\eta)\leq \big||\g|-|\eta|\big|+C.$$
\end{cor}

\begin{defi}\label{Anosov-Levi}
Given a representation $\rho:\G\to\sf G_\K$ we define its \emph{Anosov-Levi space} as $(\sf E_{\t_\rho})^*$ where $$\t_\rho=\{\aa\in\Pi:\rho\textrm{ is $\aa$-Anosov}\}.$$ It is spanned by the fundamental weights  $\{\omega_\aa:\aa\in\t_\rho\}.$
\end{defi}

A more precise description of the indicator set of $\rho$ can be given on its Anosov-Levi space. The following is a combination of Bridgeman-Canary-Labourie-S. \cite[Theorem 1.3]{pressure}, Potrie-S. \cite[Proposition 4.11]{exponentecritico} and S. \cite{exponential}.

\begin{thm}\label{Q-convex} Let $\rho:\G\to \sf G_\K$ be a representation, then $\cal Q_{\rho(\G),\t_\rho}$ is an analytic co-dimension $1$ embedded sub-manifold of $(\sf E_{\t_\rho})^*$ that varies analytically with $\rho;$ moreover its restriction to the dual of the vector space spanned by the periods is strictly convex.  
\end{thm}

\subsection{When some wall is not attained}\label{s.PS}

The purpose of this subsection is to explore $\cal Q_{\rho(\G)}$ in directions that are not controlled by the roots with respect to which $\rho$ is Anosov.

\begin{defi} 
Let $\rho:\G\to\sf G_\K$ be an $\aa$-Anosov representation. Consider $\t\subset\Pi$ with $\aa\in\t$ and let $\mu^\varphi$ be a $\big(\rho(\G),\varphi\big)$-Patterson-Sullivan measure on $\cal F_\t(\sf G_\K)$ for some $\varphi\in(\sf E_\t)^*.$ We say that $\rho$ is $\mu^\varphi$-\emph{irreducible} if for every $y\in\cal F_{\ii\t}(\sf G_\K)$ one has 
$$\mu^\varphi\big(\ann(y)\big)<\mu^\varphi\big(\cal F_\t(\sf G_\K)\big).$$ 
\end{defi}

It is clear that if $\rho(\G)$ is Zariski dense in $\sf G_\K$ then it is $\mu^\varphi$-irreducible for any Patterson-Sullivan measure. Even assuming Zariski-density, the following result is a refinement of Quint \cite[Th\'eor\`eme 8.1]{quint1} when $\t$ contains a root with respect to which $\rho$ is Anosov. Indeed, in the general case treated by Quint, one needs to control the mass of shadows on the flag space associated to $\Pi-\t$, and, as a result, the existence of a $({\rho(\G)},\varphi)$-Patterson Sullivan measure only ensures that $\varphi+\rho_{\theta^c}$ is in $\cD_{\rho(\G)}$, where $\rho_{\theta^c}$ is a suitable form that is non-negative on the Weyl chamber. In our case, the Anosov condition with respect to one root in $\t$ permits to control $\varphi$ directly. 

\begin{thm}\label{phiInD} Let $\rho:\G\to\sf G_\K$ be an $\aa$-Anosov representation. Consider $\t\subset\Pi$ with $\aa\in\t$ and let $\mu^\varphi$ be a $\big(\rho(\G),\varphi\big)$-Patterson-Sullivan measure on $\cal F_\t(\sf G_\K)$ for some $\varphi\in(\sf E_\t)^*.$ Assume $\rho$ is $\mu^\varphi$-irreducible, then $$\varphi\in\cal D_{\rho(\G)}.$$
\end{thm}

The rest of the section is devoted to the proof of this result. We begin with the following lemma from Quint \cite{quint1}. Quint assumes that the representation is Zariski dense, an hypothesis that is too strong for the appliations we have in mind. We observe however that for the proof to work only $\mu^\varphi$-irreducibility is needed. We sketch the proof for completeness.
\begin{lemma}[{\cite[Lemme 8.2]{quint1}}]\label{lemma4.4}
Let $\rho:\G\to\sf G_\K$ be a representation, $\mu^\varphi$ be a $\big(\rho(\G),\varphi\big)$-Patterson-Sullivan measure on $\cal F_\t(\sf G_\K)$. Assume $\rho$ is $\mu^\varphi$-irreducible, then there exists $\alpha_0>0$ such that for every given $0<\alpha<\alpha_0$ there exist $k>0$ only depending on $\alpha,$ such that for every $\g\in\G$ one has 
$$k^{-1}e^{-\varphi\big(\cartan(\rho(\g))\big)}\leq\mu^\varphi\Big(\rho(\g)B_{\t,\alpha}(\rho(\g))\Big)\leq ke^{-\varphi\big(\cartan(\rho(\g))\big)}.$$
\end{lemma}
\begin{proof}
Observe that $\mu^\varphi$-irreducibility guarantees that there exist $\alpha,k>0$ such that for every $\g\in\G$, $\mu^\varphi(B_{\theta,\alpha}(\rho(\g)))\geq k$: indeed otherwise there would be a sequence of reals $\alpha_n\to 0$ and elements $\g_n\in\G$ with  $\mu^\varphi(B_{\theta,\alpha_n}(\rho(\g_n)))\leq 1/n$. We can assume, up to extracting a subsequence, that the complement of $B_{\theta,\alpha_n}(\rho(\g))$ converges to $\ann(y)$ for some $y\in \cal F_{\ii \theta}$, and this contraddicts $\mu^\varphi$-irreducibility.

The result then follows from the definition of  $(\rho(\Gamma),\phi)$-Patterson-Sullivan measure using Equation (\ref{comparision-shadow}). 
 \end{proof}

The rest of the proof of Theorem \ref{phiInD} is similar to the argument showing that if there exists a Patterson-Sullivan density of a given exponent, then this exponent must be greater than the critical exponent (see for example Sullivan \cite{sullivan} and Quint's notes \cite[Theorem 4.11]{SurveyPatSul}):
\begin{proof}[Proof of Theorem \ref{phiInD}] 
We have to show that for every $s>0$ one has 
$$\sum_{\g\in\G}e^{-(1+s)\varphi\big(\cartan(\rho(\g))\big)}<\infty.$$ Corollary \ref{finite-intersectionG} implies that given $\alpha>0$ there exists $N\in\N,$ such that if $t>0$ and
 $$\G_t=\{\g\in\G: t\leq|\g|\leq t+1\},$$ then for every $x\in\bord\G$ one has 
$$\#\Big\{\g\in\G_t:\pi_\aa^{-1}(\xi_\rho^\aa (x))\cap\rho(\g)B_{\t,\alpha}(\rho(\g))\neq\emptyset\Big\}\leq N,$$
 Lemma \ref{lemma4.4} now yields for every $t\geq0$ 
\begin{equation}\label{measure-annuli}
\infty>\mu^\varphi\Big(\pi_\sroot^{-1}\big(\xi_\rho^\aa(\bord\G)\big)\Big)\geq C\sum_{\g\in\G_t}e^{-\varphi\big(\cartan(\rho(\g))\big)},
\end{equation}
where $C$ is independent of $t.$ This is to say, there exists $K>0$ independent of $t\in\R_+$ such that 
$$\sum_{\g\in\G_t}e^{-\varphi\big(\cartan(\rho(\g))\big)}<K.$$ 
Since $\varphi\in(\cal L_{\rho(\G)})^*$ and $\rho$ is $\aa$-Anosov there exist positive $\delta,\delta'$ and $C$ such that $$\varphi\big(\cartan(\rho(\g))\big)\geq \delta'\aa\big(\rho(\g)\big)\geq \delta|\g|-C.$$ One concludes that for every $s>0$ one has 

\begin{alignat*}{2}
\sum_{\g\in\G}e^{-(1+s)\varphi\big(\cartan(\rho(\g))\big)}
&\leq\sum_{n=0}^\infty\sum_{\g\in\G_n}e^{-\varphi\big(\cartan(\rho(\g))\big)}e^{-s\aa(\rho(\g))}\\ 
& \leq Ke^C\sum_{n=0}^\infty e^{-\delta s n}<\infty,
\end{alignat*} as desired.\end{proof}

\section{Anosov representations with Lipschitz limit set}\label{AnosovLipschitz}

In this section we will prove Theorem \ref{Lipschitz}. We will hence fix some notation throughout this section.

\begin{assumption}\label{basic}
The group $\G$ will be a word-hyperbolic group whose boundary $\bord\G$ is homeomorphic to a sphere of dimension $d_\G.$ We will also fix a projective Anosov representation $\rho:\G\to\PSL_d(\R)$ such that the sphere $\xi^1_\rho(\bord\G)$ is a Lipschitz submanifold of $\P(\R^d),$ i.e. it is locally the graph of a Lipschitz map. Note that we have restricted ourselves to $\K=\R.$
\end{assumption}

\subsection{The $p$-th Jacobian} Given a line $\ell$ contained in a $p+1$-dimensional subspace $V$ of $\R^d,$ the space of infinitesimal deformations of $\ell$ inside $V$ $${\sf{T}}_\ell \P(V)\subset {\sf{T}}_\ell\P(\R^d)$$ carries a natural volume form induced by the choice of a scalar product $\tau$ on $\R^d.$ Namely, if one considers the $\tau$-orthogonal decomposition $V=\ell\oplus\ell^\perp_V,$ then one canonically identifies $ {\sf{T}}_\ell \P(V)=\hom(\ell,\ell^\perp_V)$ and thus one can define $\Omega_{\ell,V}\in\wedge^{p}(\sf{T}_l\P(V))$ by 
$$\Omega_{\ell,V}(\varphi_1,\ldots,\varphi_{p})=\frac{v\wedge\varphi_1(v)\wedge\cdots\wedge\varphi_{p}(v)}{\|v\|^{p+1}}$$ 
for any $v\in\ell-\{0\}.$

\begin{defi} The linear form $\cal J^u_{p}\in(\sf E_{\{\aa_1,\aa_{p+1}\}})^*$ defined by $$\cal J^u_{p}=(p+1)\omega_1-\omega_{p+1}$$ is called \emph{the $p$-th unstable Jacobian}.
\end{defi}

\begin{lemma}\label{pJacobian}
Given $g\in\PSL_d(\R)$ and a partial flag $(\ell,V)\in\cal F_{\{\aa_1,\aa_{p+1}\}}(\R^d),$ one has 
$$g^*\Omega_{g\ell,gV}=\exp\Big(-\cal J^u_{p}\big(\bus_{\{\aa_1,\aa_{p+1}\}}(g,(\ell,V))\big)\Big)\Omega_{\ell,V}.$$\end{lemma}

\begin{proof} This is an explicit computation using equation (\ref{defBusemann}) and the definition of $\Omega_{\ell,V}.$

Indeed, whenever $\varphi_1,\ldots,\varphi_{p}\in\hom(\ell,\ell^\perp_V)$ are linearly independent, for any $v\in \ell-\{0\}$ the vectors $\{v,\varphi_1(v),\cdots,\varphi_{p}(v)\}$ form a basis of $V$ and thus:
\begin{alignat*}{2}
g^*\Omega_{g\ell,gV}(\varphi_1,\ldots,\varphi_{p})&=\Omega_{g\ell,gV}(g\varphi_1,\ldots,g\varphi_{p})\\
&=\frac{g v\wedge (g\varphi_1)(gv)\wedge\cdots\wedge(g\varphi_{p})(gv)}{\|gv\|^{p+1}}\\
&=\frac{g v\wedge g(\varphi_1(v))\wedge\cdots\wedge g (\varphi_{p}(v))}{\|gv\|^{p+1}}\\
&=\frac{g v\wedge g(\varphi_1(v))\wedge\cdots\wedge g (\varphi_{p}(v))}{v\wedge\varphi_1(v)\wedge\cdots\wedge\varphi_{p}(v)}\frac{v\wedge\varphi_1(v)\wedge\cdots\wedge\varphi_{p}(v)}{\|v\|^{p+1}}\frac {\|v\|^{p+1}}{\|gv\|^{p+1}}\\
&=\exp\big(\omega_{p+1}(\bus_{\{\aa_1,\aa_{p+1}\}}(g,V))-(p+1)\omega_1(\bus_{\{\aa_1,\aa_{p+1}\}}(g,\ell))\big)\Omega_{\ell,V}.
\end{alignat*}
\end{proof}

\subsection{Existence of a $\cal J^u_{d_\G}$-Patterson-Sullivan measure}\label{tangent} Let us prove the following proposition.

\begin{prop}\label{quasi-invariance} 
Under assumption \ref{basic}, there exists a $(\rho(\G),\cal J_{d_\G}^u)$-Patterson-Su\-lli\-van measure on $\cal F_{\{\aa_1,\aa_{d_\G}\}},$ which we will denote by $\nu_\rho.$
\end{prop}

\begin{proof} It follows from Rademacher's theorem \cite[Theorem 3.2]{EG} that $ \xi_\rho^1(\bord\G)$ has a well defined Lebesgue measure class (cfr. \cite[Section 3.2]{Fed}), and that Lebesgue almost every point $\xi^1_\rho(x)\in\xi_\rho^1(\bord\G)$ has a well defined tangent space, this defines a $d_\G+1$ dimensional vector subspace $x^{d_\G+1}_\rho\in\cal F_{\{\aa_{d_\G+1}\}}(\R^d)$ such that  \begin{equation}\label{tang}\sf T_{\xi^1_\rho(x)}\big(\xi^1_\rho(\bord\G)\big)=\hom(\xi^1_\rho(x),x^{d_\G+1}_\rho/\xi^1_\rho(x)).\end{equation}

Consider the $\rho$-equivariant measurable map $\z_\rho:\xi^1_\rho(\bord\G)\to\cal F_{\{\aa_1,\aa_{d_\G+1}\}}(\R^d)$ defined by \begin{equation}\label{measurableMap}\z_\rho(\xi^1_\rho(x))=(\xi^1_\rho(x), x^{d_\G+1}_\rho).\end{equation}

We can then define a volume form on $\xi_\rho^1(\bord\G)$ via $$\xi^1_\rho(x)\mapsto \Omega_{\z_\rho(\xi^1_\rho(x))}.$$ This form is defined Lebesgue almost everywhere and thus defines a Lebesgue measure on $\xi_\rho^1(\bord\G),$ which we will denote by $\nu_\rho.$ Lemma \ref{pJacobian} implies directly that the push-forward $(\z_\rho)_*\nu_\rho$ is the desired measure.
\end{proof}

\subsection{When $\bord\G$ is a circle}

Recall from the introduction that we say that $\rho$ is \emph{weakly irreducible} if the vector space $\spa \big(\xi^1_\rho(\bord\G)\big)$ is the whole space.

\begin{lemma}\label{replacement2} 
Under assumption \ref{basic} together with weakly irreducibility of $\rho$ and $d_\G=1,$ one has that $\rho$ is $\mu^\varphi$-irreducible for any $(\rho(\G),\varphi)$-Patterson-Sullivan measure on $\cal F_{\{\aa_1,\aa_2\}}(\R^d)$ whose projection is absolutely continuous with the measure $\nu_\rho$ constructed in Proposition \ref{quasi-invariance}.
\end{lemma}

\begin{proof}
If this were not the case, there would exist $(W_0,P_0)\in\cal F_{\{\aa_{d-2},\aa_{d-1}\}}(\R^d)$ such that $\ann(W_0,P_0)$ would have full $\mu^\varphi$-mass; as $\rho$ is projective Anosov we can furthermore assume that $P_0=\xi^{d-1}_\rho(x)$ for some $x\in\bord \G$ and thus the condition $\xi_\rho^1(y)\subset P_0$ only occurs for $y=x.$ 

Hence, since the projection of $\mu^\varphi$ is absolutely continuous with respect to $\nu_\rho$ one has that for $\mu^\varphi$-almost every $\xi^1_\rho(x) \in \xi^1_\rho(\bord\G)$ the vector space $x^{2}_\rho$ from subsection \ref{tangent} intersects $W_0.$ 

Let us choose a scalar product $\tau$ on $\R^d$, and the induced distance function of $\P(\R^d)$. Let us denote by $[W_0]$ the quotient vector space $\R^d/W_0,$ it is a 2-dimensional vector space and every line $\ell\notin W_0$ defines a line $[\ell\oplus W_0]$ in $[W_0].$ Moreover, for every $\delta>0$ the double quotient projection 
$$\pi:\big\{\ell\in\P(\R^d):\angle_\tau(\ell, W_0)>\delta\big\}\to \P\big([W_0]\big),$$ defined by $\pi(\ell)=\big[[\ell\oplus W_0]\big],$ is Lipschitz.

We denote by $U_\delta\subset \xi^1_\rho(\bord\G)$ the relative open subset defined by 
$$U_\delta=\{\ell\in\xi^1_\rho(\bord\G):\angle_\tau(\ell, W_0)>\delta\}$$ and consider the Lipschitz map $\pi|U_\delta:U_\delta\to \P([W_0])$. Since, by assumption, for $\mu^\varphi$-almost every $\xi^1_\rho(x) \in \xi^1_\rho(\bord\G)$ the plane $x^2_\rho$ intersects $W_0$, one concludes from equation (\ref{tang}) that $\pi|U_\delta$  has zero derivative $\nu_\rho$-almost everywhere.

Since Lipschitz maps are absolutely continuous, and in particular satisfy the fundamental theorem of calculus, we deduce that $\pi|\xi^1_\rho(\bord\G)$ is constant. This implies that $$\xi^1_\rho(\bord\G)\subset W_0\oplus\xi^1_\rho(x),$$ for any $x\in\bord\G,$ which contradicts the weak irreducibility assumption. 
\end{proof}

We can now prove Theorem \ref{Lipschitz} when $d_\G=1$:

\begin{cor}\label{corA1} Let $\G$ be a word-hyperbolic group such that $\bord\G$ is homeomorphic to a circle. Let $\rho:\G\to\PGL_d(\R)$ be a weakly irreducible $\aa_1$-Anosov representation such that $\xi^1_\rho(\bord\G)$ is a Lipschitz curve. Then $$\aa_1\in\cal Q_{\rho(\G)}.$$
\end{cor}

\begin{proof} Note that $\aa_1=\cal J^u_1$ is the first unstable Jacobian. Since $\xi^{1}_{\rho}(\bord\G)$ is a Lipschitz circle, it has Hausdorff dimension $1$ and thus Corollary \ref{cor:1.3} implies that $h^{\aa_1}_\rho\geq1.$

On the other hand, Proposition \ref{quasi-invariance} provides a $(\rho(\G),\cal J_1^u)$-Patterson-Sullivan measure $\mu^{\cal J_1^u}$ on $\cal F_{\{\aa_1,\aa_2\}}(V_\R)$ that projects to the Lebesgue measure on $\xi^1_\rho(\bord\G).$ 
Since $\rho$ is weakly irreducible, Lemma \ref{replacement2} implies that it is $\mu^{\cal J_1^u}$-irreducible, thus Theorem \ref{phiInD} applies to give $$\aa_1=\cal J_1^u\in\cal D_{\rho(\G)},$$ this is to say $h_\rho(\aa_1)\leq1$ which concludes the proof.

\end{proof}

Before proceeding to arbitrary $d_\G$ let us record a direct consequence of Corollary \ref{corA1}. Let us say that $\rho$ is \emph{coherent} if the first root arising in $\spa\big(\xi^1_\rho(\bord\G)\big)$ is $\sroot_1$.

\begin{cor}\label{intrinsic} Let $\G$ be a word-hyperbolic group such that $\bord\G$ is homeomorphic to a circle. Let $\rho:\G\to\sf G_\K$ be an $\aa$-Anosov representation and assume there exists proximal, real representation $\L:\sf G_\K\to\PGL(V_\R)$ with first root $\aa,$ such that $\L\circ\rho$ is coherent, then 
$$\aa\in\cal Q_{\rho(\G)}.$$
\end{cor}

\subsection{When $\bord\G$ has arbitrary dimension}

Recall that a subgroup ${{\Gamma}}\subset\PGL(V_\K)$ is \emph{strongly irreducible} if any finite index subgroup acts irreducibly. It is well known that this is equivalent to the fact that the connected component of the identity of the Zariski closure of ${{\Gamma}}$ acts irreducibly on $\K^d.$

We will need the following lemma (that does not require assumption \ref{basic}).

\begin{lemma}\label{replacement3} Let $\eta:\G\to\PGL_d(\R)$ be a strongly irreducible $\aa_1$-Anosov representation. Assume that there exists $p\in\lb1,d-1\rb$ and a measurable $\eta$-equivariant section $\z:\bord\G\to\cal F_{\{\aa_1,\aa_p\}}(\R^d).$ Then $\eta$ is $\mu^\varphi$-irreducible for any $(\rho(\G),\varphi)$-Patterson-Sullivan measure on $\cal F_\t(\K^d). $
\end{lemma}

\begin{proof}
Otherwise we would be able to find a subspace $W_0\in\cal F_{\{\aa_{d-p}\}}(\R^d)$ such that for almost every\footnote{with respect to the pushed forward measure $\pi_*\mu^\varphi,$ where $\pi:\cal F_{\{\aa_1,\aa_p\}}(\R^d)\to\P(\R^d)$ consists of forgetting the $p$-th coordinate,} $\xi^1_\rho(x)\in\xi^1_\rho(\bord \G)$ one has $\z(x)^{p}\cap W_0\neq \{0\}$. Since $\z$ is $\eta$-equivariant, we would find a $p$-dimensional subspace $V$ such that for every $\g\in\G$, $$\eta(\g)V\cap W_0\neq \{0\}.$$ This implies that for every $g$ in the Zariski closure of $\eta(\G)$ it holds that $\dim gV\cap W_0\geq 1$. The contradiction comes from Labourie \cite[Proposition 10.3]{labourie} stating that the identity component of such a Lie group cannot act irreducibly.
\end{proof}

We can now prove Theorem \ref{Lipschitz} for arbitrary $d_\G.$

\begin{cor} Under assumption \ref{basic} together with strong irreducibility of $\rho$ one has $$\cal J^u_{d_\G}\in\cal Q_\rho(\G).$$

\end{cor}

\begin{proof} Since $\xi^1_\rho(\bord\G)$ is a Lipschitz sphere, it has Hausdorff dimension $d_\G$ and thus Corollary \ref{cor:1.3} implies that $h_\rho(\cal J_{d_\G}^u)\geq 1.$ Proposition \ref{quasi-invariance} guarantees the existence of a $(\rho(\G),\cal J^u_{d_\G})$-Patterson-Sullivan measure. Moreover, the equivariant map from equation (\ref{measurableMap}) allows us to apply Lemma \ref{replacement3} and thus we are in the hypothesis of Theorem \ref{phiInD}, consequently $h_\rho(\cal J^u_{d_\G})\leq1,$ which concludes the proof.
\end{proof}
\section{$(1,1,p)$-hyperconvex representations and a $\class^1$-dichotomy for surface groups}\label{S.C^1}

In this section we will consider  projective Anosov representations whose image of the boundary map is a $\class ^1$ submanifold. In the second part of the section we will prove Corollary \ref{c.dic_intro} providing a a $\class^1$-dichotomy for surface groups.

\subsection{$(1,1,p)$-hyperconvex representations} 

\begin{defi}\label{defi:hyp}
A $\{\aa_1,\aa_p\}$-Anosov representation $\rho:\G\to\PGL_d(\R)$ is \emph{$(1,1,p)$-hyperconvex} if, for every pairwise distinct $x,y,z\in\bord\G$, the sum
$$\xi^1(x)+\xi^1(y)+\xi^{d-p}(z)$$ 
is direct.\end{defi}

\begin{ex} 
Examples of Zariski dense hyperconvex representations can be obtained by deforming $\sf S^k\circ\iota,$ where $\sf S^k$ denotes the $k$-th symmetric power and $\iota:\G\to\PO(1,p)$ is the inclusion of a co-compact lattice, see P.-S.-W. \cite[Corollary 7.6]{PSW1}. 
\end{ex}

Hyperconvex representations were introduced by Labourie \cite{labourie} for surface groups and further studied by Zhang-Zimmer \cite{ZZ} when the boundary of $\G$ is a topologically a sphere and by P.-S.-W. \cite{PSW1} for arbitrary hyperbolic groups. In both \cite[Proposition 7.4]{PSW1} and \cite[Theorem 1.1]{ZZ} one finds the following result.

\begin{thm}[P.-S.-W. and Zhang-Zimmer]\label{diff} Assume that $\bord\G$ is topologically a sphere of dimension $p-1$ and let $\rho:\G\to\PGL_d(\R)$ be a $(1,1,p)$-hyperconvex representation. Then $\xi^1_\rho(\bord\G)$ is a $\class^1$-sphere.
\end{thm}

Theorem \ref{Lipschitz} then gives:

\begin{cor}\label{cor:hyp}
Assume that $\bord\G$ is topologically a sphere of dimension $p-1$ and let  $\rho:\G\to\PSL_d(\R)$ be strongly irreducible and $(1,1,p)$-hyperconvex. Then $h_\rho(\jac^u_{p})= 1.$ \end{cor}

\begin{remark}
This generalizes Potrie-S. \cite[Corollary 7.1]{exponentecritico}. Observe however that, since the limit set $\xi^1(\bord\G)$ is a $\class^1$-submanifold of $\P(\R^d),$ the arguments of  \cite{exponentecritico} adapt directly to give a version of Corollary \ref{cor:hyp} without  requiring strong irreducibility.
\end{remark}

Glorieux-Monclair-Tholozan \cite{GMT} recently showed the following.

\begin{thm}[Glorieux-Monclair-Tholozan \cite{GMT}] Let $\rho:\G\to\PGL_d(\R)$ be an $\aa_1$-Anosov representation that preserves a properly convex domain, then $$2h_\rho(\omega_1+\omega_{d-1})\leq \Hff\big((\xi^1,\xi^{d-1})(\bord\G)\big),$$ where $(\xi^1,\xi^{d-1}):\bord\G\to\P(\R^d)\times\P\big((\R^d)^*\big).$
 
\end{thm}

As an application of Corollary \ref{cor:hyp} we show that, for $(1,1,p)$-hyperconvex representations with $p<d-1$ such bound can never be achieved:.

\begin{prop}\label{prop:h1p}
Assume that $\bord\G$ is topologically a sphere of dimension $p-1$ and let  $\rho:\G\to\PGL_d(\R)$ be strongly irreducible and $(1,1,p)$-hyperconvex. If $p<d-1$, then $$2h_\rho(\omega_1+\omega_{d-1})< (1-\epsilon) (p-1),$$ where $\epsilon>0$ only depends on the $\{\aa_1,\aa_p\}$-Anosov constants of $\rho.$
\end{prop}

\begin{proof} Since $p<d-1$ the functional $\phi\in\sf E^*$ $$\phi=\frac{\omega_p-\omega_1}{p-1}-\frac{\omega_{d-1}-\omega_1}{d-2}$$ is non-zero, moreover observe that, for every $v\in\sf E^+$ one has $$\phi(v)\geq\left(\frac{d-p-1}{d-2}\right)\aa_p(v).$$

Since $\rho$ is $\aa_p$-Anosov, the last computation implies that $\ker\phi\cap\cal L_{\rho(\G)}=\{0\}$ this is to say that $\phi\in(\cal L_{\rho(\G)})^*,$ in particular $\phi$ has a well defined entropy $h_\rho(\phi)\in(0,\infty).$ Moreover, \begin{eqnarray}\label{eq1} h_\rho\left(\frac{p-1}{d-2}((d-1)\omega_1-\omega_{d-1})\right)& = & h_\rho\big(\jac^u_{p-1}+(p-1) \phi\big) \\
 & \leq &  \frac{h_\rho(\phi)}{h_\rho(\phi)+p-1},\end{eqnarray}
where the equality comes from the equality between the corresponding linear forms and the inequality follows from Lemma \ref{entropySum} together with Corollary \ref{cor:hyp} stating that $h_\rho(\jac^u_{p-1})=1.$

Finally, observe that \begin{eqnarray*}\frac{(p-1)}2(\omega_1-\omega_{d-1}) &= &\frac12\Big(\frac{p-1}{d-2}((d-1)\omega_1-\omega_{d-1})+\frac{p-1}{d-2}((d-1)\omega_{d-1}-\omega_1)\Big)\\ & = &\frac12\big(\jac^u_{p-1}+(p-1) \phi+\big(\jac^u_{p-1}+(p-1) \phi\big)\circ\ii\big),\end{eqnarray*} where $\ii:\sf E\to\sf E$ is the opposition involution. Together with equation \ref{eq1} and Lemma \ref{entropySum}, this yields 

\begin{eqnarray*}
 \frac2{(p-1)}h_\rho(\omega_1-\omega_{d-1})& \leq & 2\frac{h_\rho\big(\jac^u_{p-1}+(p-1) \phi\big)h_\rho\big((\jac^u_{p-1}+(p-1) \phi)\circ \ii\big)}{h_\rho\big(\jac^u_{p-1}+(p-1) \phi\big)+h_\rho\big((\jac^u_{p-1}+(p-1) \phi)\circ \ii\big)} \\ 
& = & h_\rho\big(\jac^u_{p-1}+(p-1) \phi\big)\\
 & \leq &  \frac{h_\rho(\phi)}{h_\rho(\phi)+p-1}<1,\end{eqnarray*}
since entropy is $\ii$-invariant.

To conclude the proof we observe that the functional $\phi$ belongs to the Anosov-Levi space of every $\{\aa_1,\aa_p\}$-Anosov representation, its entropy thus varies continuously (Theorem \ref{Q-convex}) and hence $$\eta\mapsto \frac{h_\eta(\phi)}{h_\eta(\phi)+p-1}$$ is bounded away from 1 on compact subsets of $\frak X_{\{\aa_1,\aa_p\}}\big(\G,\PGL_d(\R)\big).$\end{proof}

\subsection*{$\class^1$-dichotomy} 
Now we prove the $\class^1$-dichotomy announced in the introduction. As we will later see (Section~\ref{sec:max} and Section~\ref{sec:positive}) there are many projective Anosov representations of surface groups where the image of the boundary map is Lipschitz.  However, when we embed the surface group into $\PSL_2(\R)$ and look small deformations of representations  
$$\G \to \PSL_2(\R) \overset{R}\to \PSL_d(\R),$$ 
where $R$ satisfies additional proximality assumptions ensuring that the representation is $\{\aa_1, \aa_2\}$-Anosov, then the image of the boundary map is never Lipschitz. 

Recall that an element $g\in\PGL_d(\R)$ is \emph{proximal} if the generalized eigenspace associated to its greatest eigenvalue (in modulus) has dimension $1.$ A representation ${R}:G\to\PGL_d(\R)$ of a given group $G$ is \emph{proximal} if its image contains a proximal element.

\begin{cor}\label{c.dic_intro} Let ${R}:\PSL_2(\R)\to\PSL_d(\R)$ be a (possibly reducible) proximal representation such that $\wedge^2{R}$ is also proximal. Let $S$ be a closed connected surface of genus $\geq2$ and let $\rho_0:\pi_1S\to \PSL_2(\R)$ be discrete and faithful. Then we have the following dichotomy: 
\begin{itemize}
\item[i)] If the top two weights spaces of ${R}$ belong to the same irreducible factor, then for every small deformation $\rho:\pi_1S\to\PSL_d(\R)$ of ${R}\rho_0$ the curve $\xi_\rho^1(\bord\pi_1S)$ is  $\class^1$. 
\item[ii)] Otherwise, for every weakly irreducible small deformation $\rho:\pi_1S\to\PSL_d(\R)$ of ${R}\rho_0$ the curve $\xi_\rho^1(\bord\pi_1S)$ is not Lipschitz. 
\end{itemize} 
\end{cor}

\begin{proof} By the proximality assumptions on ${R},$ the representation $$\rho:={R}\rho_0:\pi_1S\to\PSL_d(\R)$$ is $\{\aa_1,\aa_2\}$-Anosov.

Furthermore, if the first two weights of ${R}$ belong to the same irreducible factor, the representation $\rho$ is also $(1,1,2)$-hyperconvex, this is an open property in $\frak X\big(\pi_1S,\PSL_d(\R)\big)$ (P.-S.-W \cite{PSW1}) and thus Theorem \ref{diff} implies that every small deformation of $\rho$ has $\class^1$ limit set.

If, instead, the two top weights of ${R}$ were belonging to different irreducible factors, then it follows from the representation theory of $\SL_2(\R)$ that $$h_\rho(\aa_1)=h_\rho(\jac^u_1)=2.$$ Note that the entropy if $\jac^u_1$ is continuous on $\frak X_{\{\aa_1,\aa_2\}}\big(\pi_1S,\PSL_d(\R)\big)$ (Theorem \ref{Q-convex}), in particular there exists a neighborhood $\cal U $ of $\rho$ such that $h_\eta(\jac^u_1)>1$ for every $\eta\in\cal U.$ Theorem \ref{Lipschitz} implies that no weakly irreducible representation in $\cal U$ can have Lipschitz limit set.
\end{proof}
The regular case, item i) in Corollary \ref{c.dic_intro}, is inspired by Labourie \cite{labourie}, who treated the case (of arbitrary deformations) of the irreducible representations, and was proven in P.-S.-W \cite[Proposition 9.4]{PSW1}. The novelty of this paper is item ii),  inspired by Barbot \cite{barbot} who proved it for $d=3.$ We believe both items placed together give a clearer picture.

It is easy to obtain similar results for other group $\sf G$ by considering suitable linear representations. On the other hand the double proximality assumption is necessary: the composition of a maximal representation not in the Hitchin component and the irreducible linear representation of $\Sp(2n,\R)$ of highest weight $w_n$ is proximal but its second exterior power is not proximal; it is possible to check that no small Zariski dense deformation satisfies either (i) or (ii).

Along the same lines we can deduce that some natural Anosov representations of hyperbolic lattices do not have Lipschitz boundary maps:
\begin{cor}\label{cor:nC1}Let $\G< \PO(1,n)$ be a lattice, $n\geq 3$ and $\rho_1:\G\to\PO(1,m)$ strictly dominated by the lattice embedding $\rho_0$. Then for any Zariski dense small deformation of $\rho_0\oplus \rho_1^{n-1}$,  the limit set $\xi^1_\rho(\bord\G)$ is not Lipschitz. 
\end{cor}
Examples of lattices $\G$ admitting such representations were constructed by Danciger-Gueritaud-Kassel \cite[Proposition 1.8]{DGK18}.

\section{$\mathbb H^{p,q}$ convex-cocompact representations}\label{sec:5}
Generalizing work of Mess \cite{mess} and Barbot-M\'erigot \cite{BarbotMerigot}, Danciger--Gu\'eritaud--Kassel \cite{DGKcc} introduced a class of representations called \emph{$\bH^{p,q}$-convex cocompact}. These form another interesting class of representations with Lipschitz boundary map where Theorem \ref{Lipschitz} apply.

Let $d=p+q$ with $p,q \geq1$ and let $Q$ be a symmetric bilinear form on $\R^d$ of signature $(p,q).$ The subspace of $\P(\R^d)$ consisting on negative definite lines is called the \emph{pseudo-Riemannian hyperbolic space} and denoted by $$\bH^{p,q-1}=\{\ell\in\P(\R^d):Q|_{\ell-\{0\}}<0\}.$$ The cone of isotropic lines is usually denoted by $\bord\bH^{p,q-1}.$

Instead of the original definition of convex-cocompactness, we recall the characterization given by \cite[Theorem 1.11]{DGKcc}.

\begin{defi} An $\aa_1$-Anosov representation $\rho:\G\to\PO(p,q)$ is \emph{$\bH^{p,q-1}$-convex cocompact} if  for every pairwise distinct triple of points $x,y,z\in\bord\G$, the restriction $Q|_{\xi^1_\rho(x)\oplus\xi^1_\rho(y)\oplus\xi^1_\rho(z)}$ has signature $(2,1)$.  
\end{defi}

When $\G_0$ is a cocompact lattice in $\SO(p,1),$  $\mathbb H^{p,1}$-convex cocompact representations of $\G_0$ are usually referred to as \emph{$\AdS$-quasi-Fuchsian groups}. Barbot \cite{BarbotAds} proved that these groups form connected components of the character variety $\frak X\big(\G_0,\SO(p,2)\big)$ only consisting of Anosov representations. In \cite{GMlip} Glorieux-Monclair prove that the limit set of an $\AdS$-quasi-Fuchsian group is never a $\class^1$-submanifold, except for Fuchsian groups.

The following is well known and easy to verify, see for example Glorieux-Monclair \cite[Proposition 5.2]{GMcc}.

\begin{prop} Assume that $\bord\G$ is homeomorphic to a $p-1$-dimensional sphere.  If $\rho:\G\to\PO(p,q)$ is $\bH^{p,q}$-convex cocompact, then $\xi^1_\rho(\bord\G)$ is a Lipschitz submanifold of $\bord\bH^{p,q-1}$. 
\end{prop}

\begin{proof}
The space $\bord\bH^{p,q-1}$ admits a twofold cover that splits as the product $\bS^{p-1}\times\bS^{q-1}$. It is furthermore immediate to verify that, since for every pairwise distinct triple $(x,y,z)\in\bord\G$, $Q|_{\xi^1_\rho(x)\oplus\xi^1_\rho(y)\oplus\xi^1_\rho(z)}$ has signature $(2,1)$, each one of the two lifts of $\xi^1_\rho(\bord\G)$ to  $\bS^{p-1}\times\bS^{q-1}$ is the graph of a 1-Lipschitz function $f:\bS^{p-1}\to \bS^{q-1}$, and, as such, is a Lipschitz submanifold of $\bord\bH^{p,q-1}$.
\end{proof}

Theorem \ref{Lipschitz} then yields:

\begin{cor}\label{c.hpq}
Assume that $\bord\G$ is homeomorphic to a $p-1$-dimensional sphere and let $\rho:\G\to\PO(p,q)$ be $\bH^{p,q-1}$-convex cocompact, then 
\begin{itemize}
	\item[-] if $p=2$ and $\rho$ is weakly irreducible then $h_\rho(\jac^u_1)=1,$ 
	\item[-] if $p\geq3$ and $\rho$ is strongly irreducible then $h_\rho(\jac^u_{p-1})= 1.$ \end{itemize}
\end{cor}

One concludes the following upper bound for the entropy of the spectral radius inspired by Glorieux-Monclair \cite{GMcc}.

\begin{cor}\label{cc} Assume that $\bord\G$ is homeomorphic to a $p-1$-dimensional sphere and let $\rho:\G\to\PO(p,q)$ be $\bH^{p,q-1}$-convex cocompact. Then \begin{itemize}\item[-] if $p=2$ and $\rho$ is weakly irreducible then $h_\rho(\omega_1)\leq1,$ \item[-] for $p\geq3$ and $\rho$ strongly irreducible, $h_\rho(\omega_1)\leq p-1.$
\end{itemize}
\end{cor}

\begin{proof}
Assume first $p\leq q$ and note that for every $g\in\PO(p,q)$ one has  
$$(\omega_p-\omega_1)(\lambda(g))=\lambda_2(g)+\cdots+\lambda_p(g)\geq0.$$ By definition, $\jac^u_{p-1}=p\omega_1-\omega_p$ and thus
$$\frac{h_\rho(\omega_1)}{p-1}=h_\rho\big((p-1)\omega_1\big)\leq h_\rho(\jac^u_{p-1})=1,$$ by Corollary \ref{c.hpq}. The only difference in the case $q<p$ is that $\jac^u_{p-1}=p\omega_1-\omega_q$, but the same argument applies verbatim.
\end{proof}

The entropy for the first fundamental weight has a particular meaning for projective Anosov representations into $\PO(p,q),$ notably for $q\geq2.$ Fix $o\in\bH^{p,q-1}$ and consider $$S^o=\{W<\R^d:o\subset W,\,\dim W=q\textrm{ and }Q|W\textrm{ is negative definite}\}.$$ This is a totally geodesic embedding of the symmetric space $X_{p,q-1}$ of $\PO(p,q-1)$ in the symmetric space $X_{p,q}.$ 

Given a projective Anosov representation $\rho:\G\to\PO(p,q)$ one defines the open subset of $\bH^{p,q-1}$ $$\Omega_\rho=\{o\in\bH^{p,q-1}:Q(o,\xi^1_\rho(x))\neq0\, \forall x\in\bord\G\}.$$

Carvajales \cite{lyon} shows that, assuming $\Omega_\rho\neq\emptyset,$ for every $o\in\Omega_\rho$ one has $$\lim_{t\to\infty}\frac{\log\#\{\g\in\G:d_{X_{p,q}}(S^o,\rho(\g)S^o)\}}t=h_\rho(\omega_1)$$ and provides an asymptotic for this counting function (\cite[Theorem A]{lyon}).

When $\rho$ is moreover $\bH^{p,q-1}$-convex-cocompact, Glorieux-Monclair \cite[Section 1.2]{GMcc} introduce a \emph{pseudo-Riemannian critical exponent} $\delta_\rho,$ and show, in particular, that $$\delta_\rho\leq p-1$$ (\cite[Theorem 1.2]{GMcc}). Carvajales proves \cite[Remarks 6.9 and 7.15]{lyon} that $\delta_\rho=h_\rho(\omega_1)$ so Corollary \ref{cc} provides a different proof of \cite[Theorem 1.2]{GMcc} when $\G$ is assumed to have boundary homeomorphic to a $p-1$-dimensional sphere.

We finish the section with a direct application of Theorem \ref{infNorm} and Corollary \ref{c.hpq} allowing us to get a bound for the Riemannian critical exponent. We use freely the notation from Remark \ref{o.expSymm}.

Consider a representation $\Lambda:\PO(p,1)\to\PO(p,q)$ such that its image stabilizes a $p+1$-dimensional subspace $V$ of $\R^d$ where $Q|V$ has signature $(p,1).$ Endow the symmetric space $X_{p,q}$ with a $\PO(p,q)$-invariant Riemannian metric such that the totally geodesic copy of $\bH^p$ in $X_{p,q}$ induced be $\Lambda$ has constant curvature $-1.$ In particular, if $\iota:\G\to \PO(p,1)$ is the lattice embedding, $h_{\Lambda\circ\iota}^{X}= p-1.$
We show that this is an upper bound for any strongly irreducible, $\bH^{p,q-1}$-convex-cocompact representation:

\begin{prop}
Assume that $\bord\G$ is homeomorphic to a $p-1$-dimensional sphere and let $\rho:\G\to\PO(p,q)$ be strongly irreducible and $\bH^{p,q-1}$-convex-cocompact and, then $$h_\rho^{X}\leq p-1.$$
\end{prop}

\begin{proof} In view of Theorem \ref{infNorm} (or more precisely Remark \ref{o.expSymm}), it suffices to recall that $\cD_{\rho(\G)}$ is convex (Lemma \ref{convexD}) and that, by Corollary \ref{c.hpq}, $$\jac^u_{p-1}\in\cal Q_{\rho(\G)}.$$ See Potrie-S. \cite[Section 1.1]{exponentecritico} for more details.\end{proof}

\section{Maximal Representations}\label{sec:max}
An important class of representations that are in general only Anosov with respect to one maximal parabolic subgroup, but admit boundary maps with Lipschitz image are maximal representations into Hermitian Lie groups. In this case the Lipschitz property for the image of the boundary map is a consequence of a positivity/causality property of the boundary map. 
We first describe the causal structure on the Shilov boundary of a Hermitian symmetric space of tube type, introduce the notion of a positive curve and show that the image of any positive curve (that is not necessarily equivariant with respect to a representation)  is a Lipschitz submanifold. We then show how this applies to maximal representations and allows us to prove Theorem \ref{thm:maxgen}, the main result of this section. We also deduce consequences for the orbit growth rate on the symmetric space.

\subsection{Causal structure and positive curves}
Let $\sf G_\R$ be a simple Hermitian Lie group of tube type. Examples to keep in mind are the symplectic group $\sf G_\R=\Sp(2n,\R)$ or the orthogonal group $\sf G_\R=\SO_0(2,n)$. The Shilov boundary 
$\check S$ of the bounded domain realization of the symmetric space associated to $\sf G_\R$ is a flag variety $\sf G_\R/\check{P},$ where $\check{P}$ is a maximal parabolic subgroup determined by a specific simple root $\{\check{\aa}\}$. In the two cases that serve as our main examples, $\sf G_\R=\Sp(2n,\R)$ and  $\sf G_\R=\SO_0(2,n),$ the parabolic subgroup $\check{P}$ in question is, respectively, the stabilizer of a Lagrangian subspace $L\in\Ll(\R^{2n})$ and the stabilizer of an isotropic line $l\in\sf{Is}_1(\R^{2,n})$, so that $\check{\aa}=\sroot_n$, resp. $\check{\aa}=\sroot_1$.

In general, for a simple Hermitian Lie group of rank $n,$ there is a special set of $n$ {\em strongly orthogonal roots} $\bs_1, \cdots \bs_n$ of the complexification $\frak g_\C$, see \cite[p.582-583]{Harish-Chandra}. The set of strongly orthogonal roots give rise to a (holomorphic) embedding of a maximal polydisk. If the symmetric space is of tube type, the simple root  $\check{\aa}$ is the smallest strongly orthogonal root $\check{\aa}= \bs_n$. All the other strongly orthogonal roots are of the form $\bs_i = \bs_n+\varphi$, where $\varphi\in\sf E^*$ is non-negative on the Weyl-chamber. We record the following for later use:
\begin{lemma}\label{lem:strongly}
Let $a \in \sf{E}^+$ then ${\displaystyle \check\aa(a)= \min_{i=1,\ldots, n} \bs_i(a).}$ 
\end{lemma}

For Hermitian groups of tube type, the Shilov boundary carries a natural causal structure: for every $p\in \check S$ there is an open convex acute cone $ C_p \subset {\sf T}_p{\check S}$ which we now define.
 
Recall that $\sf G_\R/\check{P}$ can be identified as the space of parabolic subgroups of $\sf G_\R$ that are conjugate to $\check{P}$. Let us fix a point $\check{p} = \check{P} \in \check S$, which one should think of as a point at infinity.  Then at any point $p=P\in \check S$  that is transverse to $\check{p}$, i.e. such that the parabolic groups $P$ and $\check{P}$ are opposite, the tangent space ${\sf T}_p{\check S}$ is identified with the Lie algebra $\check{\frak n}$ of the unipotent radical of $\check P$, and the cone $C_p$ is an open convex acute cone $\check{C}\subset \check{\frak n}$ invariant under the action of the connected component of $P \cap \check{P}$. 

In the case of $\Sp(2n,\R)$ this is the cone of positive definite symmetric matrices, and in the case of $\SO_0(2,n)$ it is the cone of vectors with positive first entry, that are positive for the induced conformal class of Lorentzian inner products on ${\sf T}_{ P}\Is_1(\R^{2,n})$. 
 
 This invariant cone $\check C \subset \check{\frak n}$ in fact also gives rise to the notion of maximal triples in $\check S$ via the exponential map. A triple $(P,Q,\check P)$ is said to be maximal if there exists an $s\in \check{C}$ such that $Q = \exp{s} \cdot P$. Extending this by the action of $G$ leads to a notion of maximal triples in $\check S$, which actually coincides exactly with those triples which have maximal (generalized) Maslov index as introduced by Clerc-\O rsted \cite{CO}.

\begin{defi} Let $\check S$ be the Shilov boundary of a Hermitian symmetric space of tube type. A curve $\xi:\mathbb S^1\to \check S$ is \emph{positive} if the image of any positively oriented triple is a maximal triple.
\end{defi}

\begin{prop}\label{p.maxLip}
Let $\xi:\mathbb S^1\to \check{S}$ be a positive curve. Then $\xi(\mathbb S^1)$ is a Lipschitz submanifold of $\check{S}$. 
\end{prop}
\begin{proof}
Note that whenever we pick two points $p_1 = P_1, p_2= P_2$ on the image of $\xi$, the image $\xi(\mathbb S^1)$ can be covered by the two charts consisting of parabolic subgroups that are transverse to $p_1$ respectively $p_2$. 

In any of these charts  the inverse image of $\xi$, under the exponential map 
$$\begin{array}{cccc}
\frak n_i&\to &\sf G_\R/P_i\\
s&\mapsto&\exp(s)\check P_j
\end{array}$$
gives a map $\overline \xi:\R\to \frak n_i$ such that for every $t_1<t_2$ we have $\overline\xi(t_2)-\overline\xi(t_1)$ is contained in the open convex acute cone $\check C$, it then follows (see for example Burger-Iozzi-Labourie-W. \cite[Lemma 8.10]{maximalAnosov}) that the restriction of $\overline \xi$ to any bounded interval has finite length. As a result $\xi(\mathbb S^1)\subset \check S$ is rectifiable. It is thus  possible to reparametrize $\bS^1$ so that $\xi$ is a Lipschitz map.

\end{proof}

\begin{remark}
Note that we did not assume that the positive map is equivariant with respect to a representation. This will be important in Section~\ref{sec:positive}, where we will apply Proposition~\ref{p.maxLip} in this generality.\end{remark} 

\subsection{Maximal representations}
Let now $\sf G$ denote an Hermitian semisimple Lie group and let $\G$ denote the fundamental group of a closed hyperbolic surface $S.$  We consider representations $\rho:\G \to \sf G$ that are maximal, i.e. they maximize the Toledo invariant, whose definition was recalled in the introduction. Important for us is that they can be characterized in terms of boundary maps by the following theorem. 

\begin{thm}[Burger-Iozzi-W. {\cite[Theorem 8]{MaxReps}}]
A representation $\rho:\G\to \sf G$ is maximal if and only if there exists a continuous, $\rho$-equivariant, positive map $\phi:\bord\G\to\check S$.
\end{thm}

In order to apply Corollary \ref{intrinsic} we need to verify some weak irreducibility assumption. Let us first treat the case when the Zariski closure of $\rho(\G)$ is simple. 

\begin{cor}\label{cor:max}
Let $\sf G$ be a simple Hermitian Lie group of tube type and let $\check \aa$ be the root associated the Shilov boundary of $\sf G.$ If $\rho:\G\to\sf G$ is a Zariski-dense maximal representation then 
$$\check\aa\in \cal Q_{\rho(\G)},$$ this is to say, $h_\rho(\check{\aa}) = 1.$
\end{cor}
\begin{proof} Follows from Corollary \ref{intrinsic} and Proposition \ref{p.maxLip} by considering the representation $\Wedge_{\check a}$ from Proposition \ref{prop:titss}.\end{proof}	

In the remainder of this section we show how the case of maximal representations with semi-simple target group that are not necessarily Zariski-dense, can be reduced to Corollary~\ref{cor:max}. To this aim, we will use a result from Burger-Iozzi-W. \cite{tight} describing the Zariski closure $\sf H$ of a maximal representation: $\sf H$ splits as $\sf H_1 \times \cdots\times \sf H_n$,  each factor is Hermitian, and the inclusion in $\sf H\to \sf G$ is \emph{tight}. In the following we will not need the definition of a tight homomorphism, and therefore refer the interested reader to \cite[Definition 1]{tight} for it.

The following lemma will then be useful: 

\begin{lemma}\label{l.tight}
Let $\sf G$ be a classical simple Hermitian Lie group of tube type and consider a tight embedding $\iota:\sf H=\sf H_1 \times \cdots\times \sf H_n \to \sf G.$ If we denote by $\iota_*: \sf E^+_{\sf H}\to \sf E^+_{\sf G}$ the induced map, then   
$$\check{\aa}_{\sf G}\circ \iota_*=\min_{i }\check{\aa}_{\sf H_i}.$$
\end{lemma}

\begin{proof} Denote by $\pi:\frak h_1\oplus\cdots \oplus\frak h_n\to\frak g$ the associated Lie algebra homomorphism. Let $\sf E_i$ be a Cartan subspace of $\sf H_i$ and $\sf E_{\sf G}$ a Cartan subspace of $\sf G$ such that $\pi(\sf E_i)\subset \sf E_\sf G.$ 

 As $\iota$ is tight, and $\sf G$ is classical, the classification of Hamlet-Pozzetti \cite{HP} applies and gives that we have an orthogonal decomposition $\sf E_\sf G = \sf B_1 \oplus \cdots \oplus \sf B_k$ so that $\pi|\bigoplus \sf E_i$ is a direct sum of maps $\pi_i:\sf E_i\to\sf B_i$; furthermore, there are only few possibilities for the linear map $\pi_i$: if $\sf H_i$ has rank greater than one, then  $\sf B_i=\sf E_i^{m_i}$ for some $m_i$ and $\pi_i$ is a diagonal inclusion; instead,  if $\sf E_i$ is one dimensional, or equivalently $\sf H_i\cong \PSL_2(\R)$, then $\pi_i$ is induced from a direct sum of non-trivial irreducible representations (of varying degrees). It is easy to check that the subspace $\sf B_i$ is then the span of the real vectors in $\mathfrak p$ associated to the strongly orthogonal roots that do not vanish on $\pi(\sf E_i)$. Setting $\bb_i = \min_{j, \bs_j|_{\sf E_i \neq 0}} \bs_j$, we have ${\bb_i}|_{\pi(\sf E_i)} = \check{\aa}_{\sf H_i}$. And hence, with Lemma~\ref{lem:strongly}, we have $\check{\aa}_{\sf G}  = \mathrm{min}_{i }(\check{\aa}_{\sf H_i})$.\end{proof}	

We can now prove the following:
\begin{thm}\label{thm:maxgen}
Let $\sf G$ be a Hermitian semi-simple Lie group such that all factors of $\sf G$ that are of tube type are classical. Let $\theta \subset \Delta $ be the subset of simple roots associated to the Shilov boundary of $\sf G.$ Then for every maximal representation $\rho:\G\to\sf G$ one has $$\t\subset\cal Q_{\rho(\G)}.$$ \end{thm}

\begin{proof}If $\sfG = \sfG_1 \times \cdots \times\sfG_n$ then $\check S = \check S_1 \times \cdots \times\check S_n$, and therefore $\theta  = \{ \check\aa_{\sf G_1}, \cdots \check\aa_{\sf G_n}\}$ (see Burger-Iozzi-W. \cite[Lemma 3.2 (1)]{tight}). 
Furthermore $\rho:\G \to \sfG$ is maximal if and only if all $\rho_i: \G\to \sf G_i$ are maximal (Burger-Iozzi-W.  \cite[Lemma 6.1 (3)]{MaxReps}). Therefore we can restrict to the case that $\sf G$ is simple. 

Since every maximal representation factors through a representation into the normalizer of a maximal tube type subgroup $\sf H<\sf G$ (Burger-Iozzi-W. \cite[Theorem 5 (3)]{MaxReps}), which is simple, has the same rank as $\sf G$, and is such that $\check\aa_{\sf G}=\check\aa_\sf H$, we can restrict to the tube type case as the limit set in $\check S_{\sf G}$ is contained in $\check S_\sf H$ and coincides with the limit set in $\check S_\sf H$. The maximal tube type domains are always classical Hermitian symmetric spaces, except for the one exceptional Hermitian symmetric space of tube type. 
 
If now $\rho$ is not Zariski dense, then the Zariski closure is reductive and of tube type, so it is of the form $\sf H_1 \times \cdots\times \sf H_n$ and the representations into $\sf H_i$ are Zariski dense and maximal. Therefore we have $h_\rho(\check{\aa}_{\sf H_i}) = 1$ for all $i$.  As the inclusion $\sf H_1 \times \cdots\times \sf H_n\to \sf G$ is tight, the result follows from Lemma \ref{l.tight} and Lemma \ref{l.hminmaxh}.\end{proof}

\subsection{Application to the Riemannian critical exponent}
Any simple Hermitian Lie group $\sf G$ admits a diagonal embedding $\iota^\Delta:\SL_2(\R)\to \sf G$, which is equivariant with the inclusion of a diagonal disk in a maximal polydisk. We say that a representation $\rho:\G\to \sf G$ is  \emph{diagonal-Fuchsian} if it has the form $\rho=\iota^\Delta\circ \rho_0$ where $\rho_0:\G\to\SL_2(\R)$ is the lift of the holonomy of a hyperbolization.  

Let  $K_\Delta<\sf G$ be the centralizer of the image of $\iota^\Delta$, which is compact. Then a diagonal Fuchsian representation $\rho$ can be twisted by a representation $\chi:\G\to K_\Delta$. We call the corresponding representation $\rho_\chi: \G\to \sf G$ a \emph{twisted diagonal} representation. Observe that the Riemannian critical exponent $h^X$ is constant on twisted diagonal representations (the exact value $h^X_{\diag}$ depends on the choice of the normalization of the Riemannian metric)

\begin{prop}\label{prop:hmax}
Let $\G$ be the fundamental group of a closed surface and let $\rho:\G\to{\sf G}$ be a maximal representation, then $h_\rho^{X}\leq h^X_{\diag}$. 
\end{prop}

\begin{proof}
Let $\bs_1, \cdots \bs_n$ be the set of strongly orthogonal roots for $\sf G_\C$. 
It is immediate to verify that the limit cone $\cal L_{\rho_0(\G)}$ of a representation $\rho_0$ in the Fuchsian locus is concentrated in the span of the vertex of the Weyl chamber is $\sum_{i=1}^n \bs^*_i$, where $\bs^*$ is the basis of $\sf E$ dual to $\{\bs_1 \cdots \bs_n\}$. 
We know from Corollary \ref{cor:max} that, for every $\rho$, the growth rate $h_\rho(\check{\aa})=1$. Thus, if we denote by $(\sf E^+)^*$ the cone of functionals that are non-negative on the Weyl chamber, we get that $\check\sroot+(\sf E^+)^*\subset \cD_{\rho(\G)}$, and in particular all the strongly orthogonal roots are in $ \cD_{\rho(\G)}$. A simple computation shows that the  affine simplex determined by the strongly orthogonal roots meets the ray $\R\sum_{i=1}^n \bs_i$  orthogonally in a point (it is just the diagonal in a positive quadrant meeting the span of the basis vectors),  whose norm has to compute the Riemannian orbit growth rate of any representation $\rho_0$ in the Fuchsian locus: $\cal Q_{\rho_0(\G)}$ is the affine hyperplane orthogonal to 
$\R\sum_{i=1}^n \bs_i$ 
that contains $\check\sroot$. Remark \ref{o.expSymm} concludes the proof. \end{proof}

\begin{remark}
Note that when $\sf G$ is $\Sp(4,\RR)$, or more generally $\SO_\circ(2,n)$, it follows from Collier-Tholozan-Toulisse \cite{CTT} that the bound is furthermore rigid: the equality is strict unless $\rho$ is equal to $\rho_0$ up to a character in the compact centralizer of its image.

Note that for maximal representations into $\Sp(2n,\R)$ $n\geq 3$, every connected component of the space of maximal representations contains a twisted diagonal representation. However for $\Sp(4,\R)$ there are exceptional components, discovered by Gothen, where every representation is Zariski dense (see Bradlow-Garcia-Prada- Gothen \cite{BGPG} and Guichard-W. \cite{gw-topological}). In these components it is easy to verify that the bound we provide is sharp, despite not being achieved.

In the special case of the Hitchin component of $\Sp(2n,\R)$, the bound of Proposition \ref{prop:hmax} is never attained, as the irreducible representations provide a better bound that is furthermore rigid (Potrie-S. \cite{exponentecritico}).
\end{remark}

\section{$\t$-positive representations}\label{sec:positive}

Throughout this section we will write $$\sfG=\SO(p,q)$$ with $p< q$. We consider the subset $\t=\{\aa_1,\ldots, \aa_{p-1}\}$ of the simple roots discussed in Example \ref{so(p,q)} and denote by $P_\t$ the corresponding parabolic group, by $L_\t$ its Levi factor and by $U_\t$ its unipotent radical.

The group $\sfG$ admits a $\t$-\emph{positive structure} as defined by Guichard-W. \cite{GWpositivity}. This means that  for every $\sf b\in\t$ there exists an $L_\t^\rg$-invariant sharp convex cone $c_\sf b$ in $$\frak u_\sf b=\sum_{\aa\in\Sigma_\t^+,\, \aa=\sf b \text{ mod } \Span(\Pi-\t)}\frak g_\aa.$$
Here $\Sigma_\t^+=\Sigma^+\setminus\Span(\Pi-\t)$. For $\sf b\in\{\sroot_1,\ldots,\sroot_{p-2}\}$, the space $\frak u_\sf b$ is one dimensional and the sharp convex cone $c_{\sf b}=\R^+\subset \R$ consists of the positive elements, while $\frak u_{\aa_{p-1}}=\R^{q-p+2}$ endowed with a form $q_J$ of signature $(1,q-p+1)$ preserved by the action of   $\sf L_\t^0=\R^{p-2}\times \SO^0(1,q-p+1)$. The cone $c_{\aa_{p-1}}$ consists precisely of the positive vectors for $q_J$ whose first entry is positive.

Following \cite[\S 4.3]{GWpositivity} we denote by $W(\t)$ the subgroup of the Weyl group $W$ generated by the reflections $\{\sigma_i\}_{i=1}^{p-2}$ together with the longest element  $\sigma_{p-1}$ of the Weyl group $W_{\sroot_{p-1},\sroot_{p}}$ of the subroot system generated by the last two simple roots. $W(\t)$ is, in our case, a Weyl group of type $B_{p-1}$. We denote by $w_\t^0$ the longest element of $W(\t)$, and choose a reduced expression $w_\t^0=\sigma_{i_1}\ldots\sigma_{i_l}$. Of course every reflection $\sigma_i$ appears at least once among the $\sigma_{i_k}$. We consider the map
$$\begin{array}{cccc}
F_{\sigma_{i_1}\ldots\sigma_{i_l}}:&c_{\sroot_{i_1}}^\rg\times\ldots\times c_{\sroot_{i_l}}^\rg&\to&U_\t\\
&(v_1,\ldots,v_l)&\mapsto&\exp(v_1)\ldots\exp(v_l)
\end{array}$$
The $\t$-positive semigroup $U_\t^+$ is defined as the image of the map $F_{\sigma_{i_1}\ldots\sigma_{i_l}}$, and doesn't depend on the choice of the reduced expression \cite[Theorem 4.5]{GWpositivity}.

A $\theta$-positive structure on $\sfG$ gives rise to the notion of a positive triple in $\sfG/P_\t$. 
\begin{defi}
A pairwise transverse triple  in $ (\sfG/P_\t)^3$ is  \emph{$\t$-positive} if it lies in the $\sfG$-orbit of a triple of the form $(F_1,u \cdot F_1,F_3)$, where   $\Stab(F_3)=P_\t$, $F_1$ is transverse to $F_3$ and  i $u\in U_\t^+$ \cite[Definition 4.6]{GWpositivity}.
\end{defi}
\begin{remark}\label{r.positive}
The stabilizer in $\SO^0(1,q-p+1)$ of a vector $v\in c_{\aa_{p-1}}$ is compact. As a result one readily checks that the stabilizer in $\sfG$ of a $\t$-positive triple is compact.
\end{remark}
  Let now $\G_g$ be the fundamental group of a hyperbolic surface. A representation $\rho:\G_g\to \sfG$ is \emph{$\t$-positive} if there exists a $\rho$-equivariant map $\bord\G_g\to \sfG/P_\t$ sending positive triples to $\t$-positive triples \cite[Definition 5.3]{GWpositivity}. Guichard-Labourie-W. show that every $\t$-positive representation is necessarily $\t$-Anosov \cite[Conjecture 5.4]{GWpositivity}, but since the proof did not yet appear in print, in this section we will freely add this last assumption, and only discuss $\t$-positive Anosov representations.

\begin{thm}\label{p.c1}
Let $\rho:\G\to \SO(p,q)$ be $\t$-positive and $\t$-Anosov. For every $1\leq k\leq p-2$ the representation $\wedge^k\rho$ is $(1,1,2)$-hyperconvex.
\end{thm}
\begin{proof}
We denote by $\xi:\bord \G_g\to \sfG/P_\t$ the $\t$-positive continuous equivariant boundary map, and by $\xi^i:\bord\G_g\to \Is_i(\R^{p,q})$ the induced maps. By assumption, $\xi(y)=s\cdot \xi(x)$ for some element $s$ in the positive semigroup of the unipotent radical of the stabilizer of $\xi(z)$. In turn $s=\exp(v_1)\ldots\exp(v_l)$ with $v_t\in c_{\sroot_{i_t}}^\rg$ (recall that $i_t\in\{1,\ldots, p-1\}$). 
 
We set $d=p+q$. It follows from \cite[Proposition 8.11]{PSW1} that, in order to check that $\wedge^k\rho$ is $(1,1,2)$-hyperconvex, it is enough to verify that the sum
$$\xi^k_\rho(x)+\left(\xi^k_\rho(y)\cap \xi^{d-k+1}_\rho(z)\right)+\xi^{d-k-1}_\rho(z)$$
is direct, or, equivalently that the sum 
$$\xi^k_\rho(x)+s\cdot \left(\xi^k_\rho(x)\cap \xi^{d-k+1}_\rho(z)\right)+\xi^{d-k-1}_\rho(z)$$
is direct (recall that $s$ belongs to the stabilizer of $\xi_\rho(z)$). Without loss of generality we can assume that the form $Q$ defining the group $\SO(p,q)$ is represented by 
$$Q=\bpm0&0& K\\0&J&0\\K^t&0&0\epm$$
with 
$$K=\bpm0&0&(-1)^{p-1}\\ 0&\iddots&0\\-1&0&0\epm\text{ and } J=\bpm0&0&1\\ 0&-\Id_{q-p}&0\\1&0&0\epm$$
We can furthermore assume that $\xi^{l}(z)=\langle e_1,\ldots, e_l\rangle$ and $\xi^l(x)=\langle e_{d},\ldots, e_{d-l+1}\rangle$, so that $\xi^k(x)\cap \xi^{d-k+1}(z)=e_{d-k+1}$. In order to check that the representation is $(1,1,2)$-hyperconvex, we only have to verify that, given $s$ as above, writing $s\cdot e_{d-k+1}=\sum\alpha_i e_i$, the coefficient   $\alpha_{d-k}$ never vanishes. But we claim that such coefficient is just $\sum_{{i_t}=k}v_t>0$. Indeed, by construction, if $v_t\in c^\rg_{\sroot_m}$ with $m\in\{1,\ldots,p-2\}$, then $\exp(v_t)\in\SO(p,q)$ differs from the identity only in the positions $(t,t+1)$ and $(d-t, d-t+1)$ where it is equal to $v_t$ (cfr. \cite[\S 4.5]{GWpositivity}), while if $v_t\in c^\rg_{\sroot_{p-1}}$ 
we have 
$$\exp(v_t)=\bpm 
\Id_{p-2}&0&0&0&0\\
0&1&v^t&q_J(v)&0\\
0&0&\Id_{q-{p+2}}&Jv&0\\
0&0&0&1&0\\
0&0&0&0&\Id_{p-2}
\epm.$$
The result is then immediate.
\end{proof}
In particular we deduce from \cite[Proposition 7.4]{PSW1} the following
\begin{cor}\label{c.c1}
Let $\rho:\G\to \SO(p,q)$ be $\t$-positive Anosov. For every $1\leq k\leq p-2$ the image of $\xi^k_\rho(\bord\G)$ is a $\class^1$ submanifold of $\Is_k(\R^{p,q})$.
\end{cor}
We now turn to the proof of the last statement in Theorem \ref{thm:pos}. Instead of  directly verifying that the map $\xi^{p-1}_\rho$ has Lipschitz image, we will study properties of the map $\xi^{\t_0}_\rho:\bord\G_g\to \sfG/P_{\t_0}$ where $$\t_0=\{\sroot_{p-2},\sroot_{p-1}\}.$$ The flag manifold $G/P_{\t_0}$ consists of nested pairs of isotropic subspaces of dimension $p-2$ and $p-1.$
\begin{prop}\label{p.lipschitz}
Let $\rho:\G\to \SO(p,q)$ be $\t$-positive Anosov. The image of the map $\xi^{\t_0}_\rho:\bord\G_g\to \sfG/P_{\t_0}$ is a Lipschitz submanifold of $G/P_{\t_0}$.
\end{prop}
\begin{proof}
We fix a point $z\in\bord\G$ and we assume without loss of generality that $\xi^k_\rho(z)=\langle e_1,\ldots, e_k\rangle$. We denote by $\cal A\subset G/P_{\t_0}$ the set of points transverse to $\xi^{p-2,p-1}_\rho(z)$. We will show that the image of $\xi^{\t_0}_\rho|_{\bord\G\setminus \{z\}}$ is a Lipschitz submanifold of $\cal A$. Denote by $\cal A_{p-2}\subset G/P_{\sroot_{p-2}}$ the set of isotropic subspaces of dimension $p-2$ transverse to $\xi^{p-2}_\rho(z)=\langle e_1,\ldots, e_{p-2}\rangle$, by $Z_{p-1}$ the $(p-1)$-isotropic subspace $Z_{p-1}:=\xi^{p-1}_\rho(z)=\langle e_1,\ldots, e_{p-1}\rangle$, and by $Z_{p-1}^\perp$ its orthogonal with respect to the form $Q$ defining $\SO(p,q)$. Observe that we have a smooth map
$$\begin{array}{ccc}
\cal A&\mapsto &\cal A_{p-2}\times \Is_1(Z_{p-2}^\perp/Z_{p-2})\\
(Y_{p-2},Y_{p-1})&\mapsto& (Y_{p-2},[Y_{p-1}\cap Z_{p-2}^\perp])
\end{array}$$
whose image is the product of $\cal A_{p-2}$ with the set $\cal I_Z$ of isotropic lines transverse to the image of $Z_{p-1}^\perp$. Indeed for every pair $(Y_{p-2}, v)\in \cal A_{p-2}\times \cal I_Z$, the subspace $v+ Z_{p-2}$ has dimension $p-1$ and, $\dim\left( \left(v+ Z_{p-2}\right)\cap Y_{p-2}^\perp \right)=1$ as $Y_{p-2}^\perp$ and $Z_{p-2}$ are transverse. We then have $Y_{p-1}=Y_{p-2}+\left( \left(v+ Z_{p-2}\right)\cap Y_{p-2}\right)$.

We denote by $\xi_Z:\bord\G\setminus\{z\}\to \cal I_Z$ the composition of the map $\xi^{p-2,p-1}$ and the projection to the second factor in the product decomposition. The form $Q$ induces a form of signature $(2,q-p+2)$ on $Z_{p-2}^\perp/Z_{p-2}$, which gives rise to the notion of positive curves (as introduced in Section \ref{sec:max}). We claim that $\xi_Z$ is a positive curve. This amounts to showing that, if $(x,y,z)\in\bord\G$ is positively oriented, then $\xi_Z(y)=s^Z\xi_Z(x)$ for some positive element $s^Z$ in the unipotent radical of the stabilizer of $[Z_{p-1}]\in \Is_1(Z_{p-2}^\perp/Z_{p-2})$. Since the representation $\rho$ is $\t$-positive, we know that $\xi(y)=s\cdot\xi(x)$ for some element in the positive semigroup $U_\t^+$, and, as in the proof of Proposition \ref{p.c1} we can write $s=\exp(v_1)\ldots\exp(v_l)$ with $v_t\in c_{\sroot_{i_t}}^\rg$. Observe that, for every $v_t\in c_{\beta_{i_t}}^\rg$, $\exp(v_t)$ induces an element $\exp(v_t)^Z$ in the unipotent radical of the stabilizer of $[Z_{p-1}]\in \Is_1(Z_{p-2}^\perp/Z_{p-2})$, and the element $\exp(v_t)^Z$ is trivial unless $\beta_{i_t}=\sroot_{p-1}$, in which case  $\exp(v_t)^Z$ belongs to the positive semigroup of the unipotent radical of the stabilizer of $[Z_{p-1}]$. As at least one of the $v_t$ in the decomposition of $s$ belongs to such subgroup, we deduce that $\xi_Z$ is positive, as we claimed. It follows from Proposition~\ref{p.maxLip} that $\xi_Z(\bord\G\setminus\{z\})$ is a Lipschitz submanifold of $ \Is_1(Z_{p-2}^\perp/Z_{p-2})$.
 
As we know from Proposition \ref{p.c1} that $\xi^{p-2}$ is a $\class^1$-curve, we deduce that the curve $\xi^{p-2,p-1}$ is Lipschitz, being the image of a monotone map between a $\class^1$-submanifold and a Lipschitz submanifold. This concludes the proof.
\end{proof}
\subsection{The critical exponent on the symmetric space is rigid}

Let $\iota_{2p-1}:\PO(1,2)\to\PO(p,p-1)\to\PO(p,q)$ be the composition of the the irreducible representation of dimension $2p-1$ with the standard embedding of $ \PO(p,p-1)\to\PO(p,q)$. We call any representation $\rho:\G \to \PO(p,q)$, which is the composition of a Fuchsion representation with $\iota_{2p-1}$, a \emph{$(p,p-1)$-Fuchsian} representation.

\begin{lemma}\label{lem:bary}
	Let $\rho:\G\to\PO(p,q)$ be $\t$-positive Anosov. The barycenter of the affine simplex in $\sf E_\t^*$ determined by $\{\sroot_1,\ldots,\sroot_{p-2}, \epsilon_{p-1}\}$ belongs to $\cD_{\rho(\G),\t}.$ 
\end{lemma}
\begin{proof}
Recall that, in the case of $\t$-positive representations in $\PO(p,q)$, the Levi-Anosov subspace is $\sf E_{\t}:=\ker(\sroot_p)$. In particular, for every $k\leq p-2$ we have that $\sroot_k$ belongs to the dual of $\sf E_{\t}$, and belongs to the boundary of $\cD_{\rho(\G),\t}$ by Corollary \ref{c.c1}. Furthermore $\epsilon_{p-1}=\sroot_{p-1}+\sroot_p$ belongs to $\cD_{\rho(\G),\t}$ being the sum of a linear form with entropy one (the form $\sroot_{p-1}$ has entropy one by Proposition \ref{p.lipschitz}) and a linear form positive on the Weyl chamber (the root $\sroot_p$). In particular the form corresponding to the barycenter of the affine simplex they determine in $\sf E_\t^*$  belongs to $\cD_{\rho(\G),\t}.$ 
\end{proof}

\begin{thm}\label{prop:hpositive}
Let $\G$ be the fundamental group of a surface and let $\rho:\G\to\PO(p,q)$ be $\t$-positive Anosov. Then $h_\rho^{\mathcal X}\leq h_{\rho_0}^{\mathcal X}$ for any $(p,p-1)$-Fuchsian representation $\rho_0$.

If equality is achieved at a totally reducible representation $\eta$ then $\eta$ splits as $W\oplus V$  where
\begin{enumerate}
	\item $W$ has signature $(p,p-1)$ and $\eta|W$ has Zariski closure the irreducible $\PO(2,1)$ in $\PO(p,p-1)$
	\item $\eta|V$ lies in a compact group.
\end{enumerate} 
\end{thm}

\begin{proof} The inequality follows from Lemma \ref{lem:bary}, together with convexity of $\cD_{\rho(\G),\t}$ established by Theorem \ref{Q-convex}. 

Assume now that $\eta$ is a totally reducible representation such that equality holds. We can assume that $p\geq3$, as the result for $p=2$ was proven by Collier-Tholozan-Toulisse \cite[Theorem 4]{CTT}.

Let $\sf G=\overline{\eta(\Gamma)}^Z$ be its Zariski closure. By definition, $\sf G$ is a real reductive group. We consider $\sf G$ as an abstract group, denote by $\L:\sf G\to\SO(p,q)$ the inclusion representation, and by $$\phi:\frak g\to\frak{so}(p,q)$$ the associated Lie algebra morphism. Denote by $\frak a_{\sf G}$ a Cartan subspace of $\frak g.$

Since $h_\eta^{\cal X}$ attains it maximal value, Theorem \ref{Q-convex} forces the Quint indicator set $\cal Q_{\eta(\G),\t}$ to be the affine hyperplane of $(\sf E_\t)^*$ spanned by $\Delta.$ The strict convexity guaranteed by Theorem \ref{Q-convex} implies that $\sf G$ has real rank at most $2.$ Moreover we have that $\phi(\frak a_{\sf G})=\<(2(p-1),2(p-2),\ldots, 2,0), (0,\ldots,0,1)\>$.

Denote by $T=\<\xi_\eta^1(\bord\G)\>$ the vector space spanned by the projective limit curve of $\eta.$ Since $\eta$ is totally reducible, the action of $\eta(\G),$ and hence that of $\sf G,$ on $T$ is irreducible.

Fix then a Weyl chamber $\frak a_{\sf G}^+$ and let $\chi\in\frak a_{\sf G}^*$ be the highest weight of $\phi(\frak g)|T.$ Since $\eta$ is $\aa_1$-Anosov, the attracting eigenvector of every element in $\eta(\G),$ and hence of every purely loxodromic element of $\sf G,$ belongs to $V$. We therefore conclude that for every $a\in\frak a_{\sf G}^+$ $$\chi(a)=\lambda_1\big(\phi(a)\big).$$

We denote by $\cal L_\eta^{\sf G}\subset\frak a_{\sf G}^+$ Benoist's limit cone of $\eta(\G)$ in $\sf G.$ As the representation $\eta$ is $\sroot_2$-Anosov, and thus $\cal L_\eta^{\sf G}$ avoids the only wall not orthogonal to the kernel of $\sroot_1$,  there exists a linear form $\mu\in\frak a_{\sf G}^*$ such that for every $a\in\cal L_\eta^{\sf G}$ one has $$\mu(a)=\aa_1\big(\lambda\big(\phi(a)\big)\big).$$Furthermore, as $\eta$ is $(1,1,2)$-hyperconvex,  for every $x\in\bord\G$ the $2$-dimensional space $\xi^{\aa_2}(x)$ lies in $T$, and therefore  $(\chi-\mu)(a)=\lambda_2(\phi(a)),$ which implies that $\mu$ is a simple root, and $\chi=(p-1)\mu.$

For a weight $\psi$ of the representation $\phi(\frak g)|T$ or of an irreducible factor of $\phi(\frak g)|T^\perp,$ denote by $V^\psi$ the associated weight space.  We obtain from the description of $\phi(\frak a_{\sf G})$ that the weight spaces  $V^{\chi-i\mu}$ for $i\in\lb0,2p-2\rb$ are also 1-dimensional and contained in $T.$ The weight space decomposition of $T$ has thus the form $$T=\bigoplus_{i=0}^{2p-2} V^{\chi-i\mu}\oplus V^0\oplus V^{q}\oplus V^{-q},$$ where $V^0$ consists on vectors in the kernel of $\phi(\frak a_{\sf G}^+)$ (except $V^{\chi-(p-1)\mu}$) and $V^q$ corresponds to the eigenvalue $\eps_p\big(\lambda\big(\phi(a)\big)\big).$ Here, $V^0$ as well as $V^q$ and $V^{-q}$  could be instead contained in $T^\perp$, and therefore not appear in the decomposition.

Let now $W$ denote the Weyl group of $\frak g$. As the weight lattice of $\eta|T$ is $W$-invariant,  and there is no other weight of $\eta|T$ at distance $p-1$ from the origin, we deduce that $W$ is reducible, and $\frak g$ splits as $\frak g_1+\frak g_2$. If $\mu$ is the root associated to $\frak g_1$ we deduce from the fact that $V^{\chi-\mu}$ and thus $\frak g_\mu$ is one dimensional that $\frak g_1=\frak{sl} (2,\R)$. As the action of $\frak g_1$ and $\frak g_2$ commute, and the highest weight space for the restricted action of $\frak g_1$ is one dimensional, we furthermore deduce that $\frak g_2$ acts trivially on $T$. In particular $T$ is an irreducible $\frak{sl}(2,\R)$ module of dimension $2p-1$ and the signature of $T^\perp$ of the $(p,q)$-quadratic form preserved by $\frak{so}(p,q)$ is thus either negative or $(1,q-p)$. In the first case we conclude that  $\phi(\frak g)|T^\perp$ is compact. Which is the desired result.
	
In order to conclude the proof we need to exclude the second case. We know from Theorem \ref{p.c1} that for every $1\leq k\leq p-2$ and for every distinct $x,y,z\in\bord\G$ the sum
$$\xi^k(x)+\left(\xi^k(y)\cap \xi^{d-k+1}(z)\right)+\xi^{d-k-1}(z)$$
is direct. With an inductive argument we deduce that for every  $1\leq k\leq p-2$, and for every $\g\in\G$ the $k$-th eigenline belongs to $T$, and therefore the Anosov map $\xi$ would be the boundary of a Fuchsian representation composed with an embedding of $\PO(1,2)  \to \PO(p-1,p) \to \PO(p,q)$. However, such an embedding can never be positive because it has non-compact centralizer (compare Remark \ref{r.positive}). \end{proof}

\bigskip
\author{\vbox{\footnotesize\noindent 
	Beatrice Pozzetti\\
Ruprecht-Karls Universit\"at Heidelberg\\ Mathematisches Institut, Im
Neuenheimer Feld 205, 69120 Heidelberg, Germany\\
	\texttt{pozzetti@mathi.uni-heidelberg.de}
\bigskip}}

\author{\vbox{\footnotesize\noindent 
	Andr\'es Sambarino\\
	Sorbonne Universit\'e \\ IMJ-PRG (CNRS UMR 7586)\\ 
	4 place Jussieu 75005 Paris France\\
	\texttt{andres.sambarino@imj-prg.fr}
\bigskip}}

\author{\vbox{\footnotesize\noindent 
	Anna Wienhard\\
	Ruprecht-Karls Universit\"at Heidelberg\\ Mathematisches Institut, Im
Neuenheimer Feld 205, 69120 Heidelberg, Germany\\
HITS gGmbH, Heidelberg Institute for Theoretical Studies, Schloss-Wolfsbrunnenweg35, 69118 Heidelberg, Germany\\
	\texttt{wienhard@uni-heidelberg.de}
\bigskip}}

\end{document}